\documentclass[12pt,a4paper]{amsart}

\usepackage{amsmath,amsthm,amssymb}
\usepackage{float}
\usepackage{url}
\usepackage{tikz}
\usepackage{longtable}
\usepackage{array}
\newcolumntype{A}[1]{>{\rule{0pt}{4ex}}p{#1cm\arrayrulewidth-2\tabcolsep}%
	                     <{\vspace{\tabcolsep}}}%
\newcolumntype{B}[1]{>{\rule{0pt}{4ex}}m{#1cm\arrayrulewidth-2\tabcolsep}%
	                     <{\vspace{\tabcolsep}}}%

\usetikzlibrary{arrows,positioning,decorations,decorations.markings,matrix,shapes} 
\advance \textheight by -1 \topmargin
\advance \textheight by -1 \headheight
\advance \textheight by -1 \headsep
\advance \textheight by -2 \footskip
\vsize = \textheight
\hsize = \textwidth
\theoremstyle{plain}
\newtheorem{theorem}{Theorem}
\newtheorem{lemma}[theorem]{Lemma}
\newtheorem{prop}[theorem]{Proposition}
\newtheorem{corollary}[theorem]{Corollary}
\newtheorem{definition}[theorem]{Definition}
\theoremstyle{remark}
\newtheorem{remark}[theorem]{Remark}
\numberwithin{theorem}{section}

\newcommand{\A}{\mathbb{A}}
\newcommand{\B}{\mathbb{B}}
\newcommand{\D}{\mathbb{D}}
\newcommand{\E}{\mathbb{E}}

\newcommand{\HH}{\mathbb{H}}
\newcommand{\Q}{\mathbb{Q}}
\newcommand{\Z}{\mathbb{Z}}

\newcommand{\SN}{\theta}
\newcommand{\Aut}{\textnormal{Aut}}
\newcommand{\PSL}{\textnormal{PSL}}

\newcommand{\GL}{\textnormal{GL}}

\newcommand{\GO}{\textnormal{O}}

\newcommand{\Nr}{\textnormal{Nr}}
\newcommand{\SL}{\textnormal{SL}}
\newcommand{\Sp}{\textnormal{Sp}}
\newcommand{\PSp}{\textnormal{PSp}}
\newcommand{\Stab}{\textnormal{Stab}}
\newcommand{\MO}{\mathcal{M}}
\newcommand{\isom}{\cong}
\newcommand{\tensor}{\otimes}

\newcommand{\Syl}{\textnormal{Syl}}
\newcommand{\Alt}{\textnormal{Alt}}

\newcommand{\fraka}{\mathfrak{a}}

\newcommand{\calB}{\mathcal{B}}
\newcommand{\calC}{\mathcal{C}}
\newcommand{\calG}{\mathcal{G}}

\newcommand{\calO}{\mathcal{O}}
\newcommand{\calQ}{\mathcal{Q}}
\newcommand{\calL}{\mathcal{L}}
\newcommand{\calD}{\mathcal{D}}

\newcommand{\frakp}{\mathfrak{p}}
\newcommand{\frakq}{\mathfrak{q}}
\newcommand{\frakA}{\mathfrak{A}}
\newcommand{\frakP}{\mathfrak{P}}
\newcommand{\frakQ}{\mathfrak{Q}}
\newcommand{\U}{\textnormal{U}}
\newcommand{\SU}{\textnormal{SU}}

\newcommand{\genus}{\textnormal{genus}}

\makeatletter
\newcommand{\bigperp}{%
  \mathop{\mathpalette\bigp@rp\relax}%
  \displaylimits
}

\newcommand{\bigp@rp}[2]{%
  \vcenter{
    \m@th\hbox{\scalebox{\ifx#1\displaystyle2.1\else1.5\fi}{$#1\perp$}}
  }%
}
\makeatother

\begin{document}

\bibliographystyle{plain}

\title{One class genera of lattice chains over number fields}
\author{Markus Kirschmer}
\email{markus.kirschmer@math.rwth-aachen.de}
\author{Gabriele Nebe}
\email{nebe@math.rwth-aachen.de}
\address{Lehrstuhl D f\"ur Mathematik, RWTH Aachen University, 52056 Aachen, Germany}
%\date{}

\begin{abstract}
We classify all one-class genera of admissible 
lattice chains of length at least $2$ in hermitian spaces over number fields.
\end{abstract}

\maketitle

\section{Introduction}

Kantor, Liebler and Tits \cite{KLT} 
classified discrete groups $\Gamma $ with a type preserving chamber transitive action
on the affine building $\calB ^+$ of a simple adjoint algebraic group
of relative rank $r\geq 2$. 
Such groups are very rare and hence this situation is
an interesting phenomenon.
Except for two cases in characteristic 2 (\cite[case (v)]{KLT}) and 
the exceptional group $\textnormal{G}_2(\Q_2) $ (\cite[case (iii)]{KLT}, Section \ref{g2}) 
the groups arise from classical groups $\U_p$  over $\Q_p$ for $p=2,3 $. 
Moreover $\Gamma $ is a subgroup of the $S$-arithmetic group 
$\Gamma _{max} := \Aut (L \otimes _{\Z } \Z [\frac{1}{p} ]) $  (so $S = \{ p \}$) 
for a suitable lattice $L$ in some hermitian space $(V,\Phi )$ and 
$\U_p = \U (V_p,\Phi )$ is the completion of the unitary group $\U(V,\Phi )$ 
(see Remark \ref{completiongroup}). 
This paper uses  the classification of one-  and two-class genera of hermitian lattices
in \cite{Kirschmer} to obtain 
these $S$-arithmetic groups  $\Gamma _{max} $.

Instead of the thick building $\calB^+$ 
we start with  the affine building $\calB$ of admissible lattice chains 
as defined in \cite{AN}.  
The points in the building ${\mathcal B}$  correspond to 
homothety classes of certain $\Z _p$-lattices in $V_p $.
The lattices form a simplex in ${\mathcal B} $, if and only if 
representatives in these classes can be chosen to form an 
admissible chain of lattices in $V_p$. 
In particular the maximal simplices of ${\mathcal B} $ (the so called chambers) 
correspond to the fine admissible lattice 
chains in $V_p$ (for the thick building 
$\calB ^+ $ one might have to apply the oriflamme construction
as explained in Remark \ref{oriflamme}).

Any fine admissible lattice chain $\calL _p$ in $V_p$ arises as 
the completion of a  lattice chain $\calL '$ in $(V,\Phi )$.  
After rescaling and applying the reduction operators from Section \ref{normalizedgenera} 
we obtain a fine $p$-admissible lattice chain 
${\mathcal L} = (L_0,\ldots , L_r )$  in 
$(V,\Phi )$ (see Definition \ref{admissible}) such that 
$\Aut ({\mathcal L}) \supseteq \Aut (\calL ') $ and such that 
the completion of $\calL$ at $p$ is $\calL_p$. 
The $S$-arithmetic group $\Aut(L_0 \otimes \Z[\frac{1}{p} ] ) = \Aut(L_i \otimes \Z[\frac{1}{p}]) 
 =: \Aut(\calL \otimes \Z[\frac{1}{p}] ) $  contains $\Aut (\calL' \otimes \Z[\frac{1}{p}] )$. 
Therefore we call this group closed. 

The closed $\{ p \} $-arithmetic  group  $\Aut(L_0 \otimes \Z[\frac{1}{p} ] ) $ 
acts chamber transitively on ${\mathcal B}$, if the  lattice $L_0$ represents a genus of class 
number one 
and $\Aut (L_0 )$ acts transitively on the fine flags 
of (isotropic) subspaces in the hermitian space $\overline{L_0}$ (see Theorem \ref{main}). 
If we only impose chamber transitivity on 
the thick building $\calB ^+$, then 
 we also have to take two-class genera of lattices $L_0$ into account. 
To obtain a complete classification of all chamber transitive actions of 
closed $S$-arithmetic groups on the thick building $\calB ^+$ 
we would have to apply our strategy to the still unknown classification of 
proper special genera of lattices $L_0$ with class number one (see Theorem \ref{classtwo}). 

Already taking only the one-class genera of lattices $L_0$ we find 
all the groups from \cite{KLT} and one additional case (described in Proposition \ref{hf} (1)).
 Hence our computations 
confirm and supplement the classification of \cite{KLT}. 
A list of the corresponding buildings and groups $\U_p $ is 
given in Section \ref{buildtable}.

\section{Lattices in hermitian spaces} 

Let $K$ be a number field. 
Further, let $E/K$ be a field extension of degree at most $2$ or let $E$ be a quaternion skewfield over $K$. 
The canonical involution of $E/K$ will be denoted by $\sigma \colon E \to E$. In particular, $K$ is the fixed field of 
$\sigma$ and hence the involution $\sigma$ is the identity if and only if $K=E$. 
A hermitian space over $E$ is a finitely generated (left) vector space $V$ over $E$ equipped with a non-degenerate
sesquilinear form $\Phi \colon V \times V \to E$ such that

\begin{itemize}
  \item $\Phi(x+x', y) = \Phi(x,y) + \Phi(x', y)$ for all $x,x',y \in V$.
  \item $\Phi(\alpha x, \beta y) = \alpha \Phi(x,y) \sigma(\beta)$ for all $x,y \in V$ and $\alpha, \beta \in E$.
  \item $\Phi(y,x) = \sigma(\Phi(x, y))$ for all $x,y \in V$.
\end{itemize}

The unitary group 
$\U (V,\Phi )  $ 
 of $\Phi $ is the group of all $E$-linear endomorphisms of $V$ that 
preserve the hermitian form $\Phi$.
The {\em special unitary group} is defined as 
$$\SU(V,\Phi ):=  \{ g\in \U(V,\Phi ) \mid \det(g) = 1 \} $$ 
if $E$ is commutative and $\SU(V,\Phi ):= \U(V,\Phi) $ if $E$ is a 
quaternion algebra. 

We denote by $\Z_K$ the ring of integers of the field $K$ and we fix some maximal order $\MO$ in $E$. Further, let $d$ be the dimension of $V$ over $E$.

\begin{definition}
	An {\em $\MO $-lattice} in $V$ is a finitely generated 
	$\MO $-submodule of $V$ that contains an $E$-basis of $V$. 
	If $L$ is an $\MO $-lattice in $V$ then its {\em automorphism group} 
	is 
	$$\Aut(L):= \{ g\in \U(V,\Phi ) \mid Lg = L \} .$$
	%Similarly for any ring $R$ with $\MO \subseteq R \subseteq E$
	%and any $R$-lattice $L$ in $V$ we define 
	%$\Aut(L) $ as the stabiliser in $\U(V,\Phi )$ of $L$ and 
	%the {\em proper automorphism group} 
	%$$\Aut^+(L) := \{ g\in \Aut(L) \mid \det(g) = 1\}.$$
\end{definition}

\subsection{Completion of lattices and groups} 

Let $\frakP$ be a maximal two sided ideal of $\MO$ and let $\frakp = \frakP \cap K$. 
The completion $\U_{\frakp} := \U(V\otimes _K K_{\frakp } ,\Phi )$ is an algebraic group over the $\frakp$-adic completion $K_{\frakp }$ of $K$. 

Let $L\leq V$ be some $\MO $-lattice in $V$. 
We define the $\frakp $-adic completion of $L$ as $L_{\frakp }:=L\otimes _{\Z_K} \Z_{K_{\frakp}} $ and  we let
 $$L(\frakp):=\{ X\leq V \mid X_{\frakq} = L_{\frakq } \mbox{ for
 all prime ideals } \frakq \neq \frakp \}, $$
  be the set of all $\MO $-lattices in $V$ whose $\frakq $-adic 
 completion coincides with the one of $L$ for all prime ideals $\frakq \neq \frakp$. 

 \begin{remark} \label{localglobal} 
	 By the local global principle, given a lattice $X$ in $V_{\frakp }$,
	 there is a unique lattice $M \in L(\frakp) $ with $M_{\frakp} = X $.
 \end{remark}

To describe the groups $\U_{\frakp}$ in the respective cases, 
we need some notation: 
Let $R$ be one of $E,K,\Z_K ,\MO $ or a suitable completion. 
A hermitian module $\HH (R)$ with $R$-basis $(e,f)$ satisfying 
$\Phi (e,f) = 1 , \Phi(e,e) = \Phi(f,f) = 0 $ is called a 
{\em hyperbolic plane}.
By \cite[Theorem (2.22)]{Kneser} any  hermitian space  over  $E$ 
is either anisotropic (i.e. 
$\Phi (x,x) \neq 0 $ for all $x\neq 0$) or has a hyperbolic plane as an orthogonal
direct summand.

\begin{remark}\label{completiongroup} 
In our situation the following cases are possible:
\begin{itemize}
	\item $E=K$: Then $(V\otimes _K K_{\frakp } ,\Phi ) $ is a quadratic space and hence isometric to 
		$\HH (K_{\frakp } ) ^r \perp (V_0,\Phi _0)$ with $(V_0,\Phi_0)$ anisotropic. 
		The rank of $\U_{\frakp} $ is $r$.  
		The group that acts type preservingly on the 
		thick Bruhat-Tits building $\calB ^+$  defined in Section \ref{Secoriflamme} is 
		$$\U^+_{\frakp }:= \{ g\in \U_{\frakp} \mid \det (g) = 1, \SN (g) \in K^2 \} $$
		the subgroup of the special orthogonal group with trivial spinor norm $\theta$. 
	\item $\frakP \neq \sigma(\frakP)$. Then $E\otimes _K K_{\frakp} \cong K_{\frakp} \oplus K_{\frakp }$ 
		where the involution interchanges the two components and $\U_{\frakp} \cong \GL _d (K_{\frakp}) $ has
		rank $r=d-1$.
As $\frakP$ is assumed to be a maximal $2$-sided ideal of $\MO $, the case that 
		$E$ is a quaternion algebra is not possible here. 
		Here we let $$\U^+_{\frakp } = \{ 
			 g\in \U_{\frakp} \mid \det (g) = 1 \} 
			 = \SL _d(K_{\frakp}) .$$
	\item $[E:K] =4$ and $\frakP =\frakp \MO$. Then $E_\frakp \isom K_\frakp^{2 \times 2}$ and for $x \in E_\frakp$,  $\sigma(x)$ is simply the adjugate of $x$ as $\sigma(x) x \in K$.
		Let $e^2=e \in E_{\frakp}  $ such that  $\sigma(e) = 1-e$. Then
	      $V_\frakp = e V_\frakp \bigoplus (1-e) V_\frakp$. 
	      The hermitian form $\Phi$ gives rise to a skew-symmetric form 
	      \begin{align*}  \Psi \colon eV_\frakp \times eV_\frakp &\to e E_\frakp (1-e) \isom K_\frakp,\\ (ex, ey) &\mapsto \Phi(ex,ey) = e \Phi(x,y) (1-e) \:. \end{align*}
	      From $E_\frakp = E_\frakp e E_\frakp$ we conclude that $V_\frakp = E_\frakp e E_\frakp V$.
	      Hence we can recover the form $\Phi$ from $\Psi$ and thus $\U_\frakp \isom \U(eV, \Psi) \isom \Sp_{2d}(K_\frakp)$ has rank $r=d$. 
	      Here the full group $\U_{\frakp }$ acts type preservingly on 
	      $\calB^+$ and we put $\U^+_{\frakp } := \U_{\frakp } $.
	\item %$[E: K]=2$ and $\frakP = \sigma(\frakP )$ 
	      In the remaining cases $E\otimes K_{\frakp } = E_{\frakP }$ is a skewfield, 
		which is ramified over $K_{\frakp} $ if and only if $\frakP ^2 = \frakp \MO $. 
		In all cases $\U_{\frakp} $ is isomorphic to a unitary group over $E_{\frakP} $. 
		Hence it admits a decomposition $\HH(E_\frakP)^r \perp (V_0,\Phi _0)$ with $(V_0,\Phi_0)$ anisotropic
		where $r$ is the rank of $\U_{\frakp} $. If
		$E_{\frakp } $ is commutative, we define 
		$$\U^+_{\frakp }:= \{ g \in \U_{\frakp } \mid \det(g) =1 \} 
		= \SU_{\frakp } $$
		and put $\U^+_{\frakp } = \SU _{\frakp} := \U _{\frakp }$ in the non-commutative
		case.
\end{itemize} 
\end{remark}

\subsection{The genus of a lattice} 

To shorten notation, we introduce the adelic ring $A= A(K) = \prod_v K_v$ where
$v$ runs over the set of all places of $K$.
We denote the  adelic unitary group of the
$A\tensor _K E$-module $V_A = A \tensor_K V$ by $\U (V_A, \Phi )$. 
The normal subgroup 
$$\U ^+(V_A , \Phi ) := \{ (g_{\frakp } ) _{\frakp} \in \U (V_A, \Phi ) \mid 
g_{\frakp } \in \U^+_{\frakp } \} \leq \U (V_A, \Phi ) $$
is called the special adelic unitary group.

The adelic unitary group acts on the set of all $\MO $-lattices in $V$ 
by letting $Lg = L'$ where $L'$ is the unique lattice in $V$ such that 
its $\frakp $-adic completion $(L')_\frakp = L_{\frakp} g_{\frakp}$
for all maximal ideals $\frakp $ of $\Z _K$.

\begin{definition}
Let $L$ be an $\MO$-lattice in $V$. Then
\[ \genus(L):= \{ Lg \mid g \in \U (V_A, \Phi )  \} \]
is called the \emph{genus} of $L$. \\
Two lattices $L$ and $M$ are said to be \emph{isometric} (respectively {\em properly isometric}), if $L = Mg$ for some $g \in \U(V,\Phi)$
(resp. $g\in \SU(V,\Phi ) $).
\\
Two lattices $L$ and $M$  are said to be in the same 
{\em proper special genus}, 
if there exist $g\in \SU(V,\Phi )$ and $h\in \U^+(V_A,\Phi )$ 
such that $Lgh = M$.
\end{definition}

Let $L$ be an $\MO$-lattice in $V$.
It is well known that $\genus(L)$ is a finite union of isometry classes, c.f.  \cite[Theorem 5.1]{Borel_finite}.
The number of isometry classes in $\genus(L)$ is called the class number 
$h(L) $ of (the genus of) $L$.
Similarly the proper special genus is a finite union of proper isometry
classes, the proper class number will be denoted by $h^+(L)$.

\subsection{Normalized genera} \label{normalizedgenera}

\begin{definition}
Let $L$ be an $\MO$-lattice in $V$. Then $L^\# = \{x \in V \mid \Phi(x,L) \subseteq \MO \}$ is called the \emph{dual} lattice of $L$.
If $\frakp$ is a maximal ideal of $\Z_K$, then the unique \mbox{$\MO$-lattice} $X
\in L(\frakp )$ such that
$ X_\frakp = L_\frakp^\# $
is called the \emph{partial dual of $L$ at $\frakp$}. It will be denoted by $L^{\#, \frakp}$.
\end{definition}

\begin{definition}\label{squarefree}
Let $L$ be an $\MO$-lattice in $V$.
Further, let $\frakP$ be a maximal two sided ideal of $\MO$ and set $\frakp = \frakP \cap K$.
If $E_\frakp \isom K_\frakp \oplus K_\frakp$ then $L_\frakp$ is called \emph{square-free} if $L_\frakp = L_\frakp^\#$.
In all other cases, $L_\frakp$ is called square-free if $\frakP L_\frakp^\#  \subseteq L_\frakp \subseteq L_\frakp^\#$.
The lattice $L$ is called square-free if $L_\frakp$ is square-free for all maximal ideals $\frakp$ of $\Z_K$.
\end{definition}

Given a maximal two sided ideal $\frakP$ of $\MO$, we define an operator $\rho_\frakP$ on the set of all $\MO$-lattices as follows:
\[ \rho_\frakP(L) = 
\begin{cases}
L + (\frakP^{-1} L \cap L^\#) & \text{if }\frakP \ne \sigma(\frakP),\\
L + (\frakP^{-1} L \cap \frakP L^\#) & \text{otherwise.}
\end{cases} \]
The operators generalize the maps defined by L. Gerstein in \cite{Gerstein} for quadratic spaces.
They are similar in nature to the \emph{$p$-mappings} introduced by G.~Watson in \cite{pMap}.
The maps satisfy the following properties:

\begin{remark}\label{RhoRemark} Let $L$ be an $\MO$-lattice in $V$. Let $\frakP$ be a maximal two sided ideal of~$\MO$ and set $\frakp= \frakP \cap \Z_K$.
\begin{enumerate}
	\item  $\rho_\frakP(L) \in L(\frakp )$.
\item If $L_\frakp$ is integral, then $(\rho_\frakP(L))_\frakp = L_\frakp \iff L_\frakp$ is square-free.
\item If $\frakQ$ is a maximal two sided ideal of $\MO$, then $\rho_{\frakP} \circ \rho_\frakQ = \rho_{\frakQ} \circ \rho_\frakP$.
\item If $L$ is integral, there exist a sequence 
	of not necessarily distinct
	maximal two sided ideals $\frakP_1, \dots, \frakP_s$ of $\MO$ such that
\[ L':= (\rho_{\frakP_1} \circ \ldots \circ \rho_{\frakP_s})(L) \]
is square-free. Moreover, the genus of $L'$ is uniquely determined by the genus of $L$.
\end{enumerate}
\end{remark}
 
\begin{prop}
Let $L$ be an  $\MO$-lattice in $V$ and let $\frakP$ be a maximal two sided ideal of~$\MO$.
Then the class number of $\rho_\frakP(L)$ is at most the class number of $L$.
\end{prop}
\begin{proof}
The definition of $\rho_\frakP(L)$ only involves taking sums and intersections of multiples of $L$ and its dual.
Hence $ \rho_\frakP(L)g = \rho_\frakP (Lg) $ for all $g \in \U(V, \Phi)$ and similar for $g \in \U(V_A, \Phi)$.
In particular, $\rho_\frakP$ maps lattices in the same genus (isometry class) to ones in the same genus (isometry class).
The result follows.
\end{proof}

\begin{definition}
Let $\frakA$ be a two sided $\MO$-ideal.
An $\MO$-lattice $L$ is called $\frakA$-maximal, if $\Phi(x,x) \in \frakA$ for all $x \in L$ and no proper overlattice of $L$ has that property.
Similarly, one defines maximal lattices in $V_\frakp$ for a maximal ideal $\frakp$ of $\Z_K$.
\end{definition}

\begin{definition}\label{normalised} 
Let $\frakP$ be a maximal two sided ideal of~$\MO$ and set $\frakp = \frakP \cap K$.
We say that an $\MO $-lattice 
$L$ is  $\frakp$-\emph{normalized} if $L$ satisfies the following conditions:
\begin{itemize}
\item $L$ is square-free.
\item If $E=K$ then $L_\frakp \isom \HH(\Z_{K _\frakp})^r \perp M_0$ 
	where $M_0 = \rho_\frakp^\infty(M)$ and $M$ denotes a \mbox{$2\Z_{K_\frakp}$-maximal} lattice in an anisotropic quadratic space over $K_\frakp$.
\item If $E_\frakp/K_\frakp$ is a quadratic field extension with different $\calD(E_\frakp/K_\frakp)$, 
then $L_\frakp \isom \HH(\MO _\frakp)^r \perp M_0 $ where $M_0 = \rho_\frakp^\infty(M)$ and $M$ denotes a $\calD(E_\frakp/K_\frakp)$-maximal lattice in an anisotropic hermitian space over $E_\frakp$.
\item If $[E:K] = 4$, then $L_\frakp = L_\frakp^\#$.
\end{itemize}
Here $\rho_\frakP^\infty(M)$ denotes the image of $M$ under repeated application of $\rho_\frakP$ until this process becomes stable.
\end{definition}

\begin{remark}
Let $\frakP, \frakp$ and $L$ be as in Definition \ref{normalised}. Then the isometry class of $L_\frakp$ is uniquely determined by $(V_\frakp,\Phi)$.
\end{remark}
\begin{proof}
There is nothing to show if $[E:K]=4$. Suppose now $E=K$.
The space $K M_0$ is a maximal anisotropic subspace of $(V_\frakp, \Phi)$. By Witt's theorem \cite[Theorem 42:17]{OMeara} its isometry type is uniquely determined by $(V_\frakp, \Phi)$.
Further, $M_0$ is the unique $2\Z_{K_\frakp}$-maximal $\Z_{K_\frakp}$-lattice in $K M_0$, see \cite[Theorem 91:1]{OMeara}.
Hence the isometry type of $\rho_\frakp^{\infty}(M_0)$ depends only on $(V_\frakp,\Phi)$.
The case $[E:K]=2$ is proved similarly.
\end{proof}

\section{Genera of lattice chains} 

\begin{definition}
	Let $\calL := (L_1,\ldots , L_m )$ and $\calL' := (L'_1,\ldots ,L'_m)$  
	be two $m$-tuples of $\MO$-lattices in $V$. 
	Then  $\calL $ and $\calL'$ are 
	{\em isometric}, if there is some $g\in \U (V,\Phi )$ 
	such that $L_ig = L_i' $ for all $i=1,\ldots,m $. 
	They are in the same {\em genus} 
	if there is such an element $g\in \U (V_A,\Phi )$. 
	Let 
	$$[\calL ] := \{ \calL' \mid \calL' \mbox{ is isometric  to } \calL \} $$
	and 
$$\genus (\calL) :=  \{ \calL' \mid \calL' \mbox{ and  } \calL \mbox{ are in the same genus}  \} $$
denote the isometry class and the genus of $\calL $, respectively.
	The {\em automorphism group} of $\calL $ is the stabilizer of $\calL $ 
	in $\U(V,\Phi )$, i.e.
	$$\Aut (\calL ) = \bigcap _{i=1}^m \Aut (L_i ) .$$
\end{definition}

It is well known \cite[Theorem 5.1]{Borel_finite}  that any genus of a single lattice 
contains only finitely many isometry classes. 
This is also true for finite tuples of lattices in $V$:

\begin{lemma} 
	Let $\calL = (L_1,\ldots , L_m )$ be an  $m$-tuple of $\MO$-lattices in $V$.
	Then $\genus (\calL) $ is the disjoint union of finitely 
	many isometry classes.
	The number of isometry classes in $\genus(\calL ) $ is 
	called the class number of $\calL $.
\end{lemma} 

\begin{proof}
	The case $m=1$ is the classical case.
	So assume that $m\geq 2$ and let 
	$\genus (L_1) := [M_1] \uplus \ldots \uplus [M_h] $, 
	with $M_i =L_1 g_i $ for suitable $g_i \in \U(V_A,\Phi )$.
	We decompose 
	$\genus (\calL ) = \calG _1 \uplus \ldots  \uplus \calG _h $ 
	where 
	$\calG _i := \{ (L'_1,\ldots ,L'_m) \in \genus (\calL ) \mid 
	L'_1 \cong M_i \} $. 
	It is clearly enough to show that  each $\calG _i$
	is the union of finitely many isometry classes. 
By construction, any isometry class in $\calG _i$ contains a representative 
of the form $(M_i,L'_2,\ldots,L'_m )$ for some lattices $L'_j$ in the
genus of $L_j$.
	As all the $L_j$ are lattices in the same vector space $V$, there are
 $a,b \in \Z_K$ such that 
 $$ b L_1 \subseteq L_j \subseteq \frac{1}{a} L_1 \mbox{ for all } 1\leq j \leq m.$$ 
As 
$(M_i,L'_2,\ldots,L'_m ) =  \calL g$ for some $g\in \U(V_A,\Phi )$ 
we also have 
$$ b M_i \subseteq L'_j \subseteq \frac{1}{a} M_i \mbox{ for all } 2\leq j \leq m .$$
So there are only finitely many possibilities for such lattices $L'_j$.
Hence 
the set of all $m$-tuples $(M_i,L'_2,\ldots,L'_m ) \in \genus (\calL )$ 
is finite and so is the class number. 
\end{proof}

%From now on, we assume that $\Phi $ is a (totally positive) definite hermitian form on $V$. This means that 
%$K$ is totally real and $\Phi(x,x)$ is totally positive for all nonzero $x \in V$.
%Then for any $\MO$-lattice $L$ in $V$, the group $\Aut (L) $ is finite.
%
%\begin{definition}
	%The {\bf mass} $\mass(\calL )$ of 
	%a finite tuple $\calL $ of lattices in $(V,\Phi )$ 
%is defined as 
%$$\mass (\calL ):= \sum _{i=1}^h |\Aut (\calL_i)| ^{-1} $$
%where $(\calL _i)_{i=1}^h $ is a system of representatives of the isometry classes in
%$\genus(\calL)$.
%\end{definition}

\begin{remark}
			If $\calL ' \subseteq \calL $ then 
			the class number of 
			$\calL' $ is at most the class number of $\calL$. 
\end{remark}

\subsection{Admissible lattice chains} 

\begin{definition}\label{admissible}
Let $\frakP$ be a maximal 2-sided ideal  of $\MO$ and $\frakp := K \cap \frakP$.
A lattice chain   
$$ \calL : = \{  L_0\supset L_{1} \supset \ldots \supset L_{m-1} \supset L_m \} $$
is called {\em admissible} for $\frakP$, if 
\begin{enumerate}
	\item $L_0 \subseteq L_0^{\#, \frakp} $,
	\item $ \frakP L_0 \subset L_m $, 
	\item $\frakP L_m^{\#, \frakp} \subseteq L_m$ if $\frakP = \sigma(\frakP) $. 
\end{enumerate}
%For an admissible chain $\calL$ we define 
%$$\begin{array}{ll}  
	%a_i:= \dim (L_{i-1}/L_{i})  & 1\leq i \leq m  \\ 
%a_0:= \dim (L_0^{\#,\frakp } /L_0) \end{array} $$
%and call $ (a_0,a_1,\ldots , a_{m}) $ the {\bf type} of ${\calL}$.
%\\
We call a $\frakP$-admissible chain {\em fine}, if 
$L_0$ is normalized for $\frakp$ in the sense of Definition \ref{normalised}, 
$L_i$ is a maximal sublattice of $L_{i-1}$ for all $i=1,\ldots , m$ 
and 
either 
\begin{itemize}
	\item[(a)] $\frakP = \sigma (\frakP )$ and 
 $L_m/\frakP L_m^{\#, \frakp}$ is an anisotropic space over $\MO/\frakP $ 
\item[(b)] $\frakP \neq \sigma(\frakP) $ and 
 $\frakP L_0 $ is a maximal sublattice of $L_m$.
	\end{itemize}
\end{definition}

\begin{remark}
	In the case that $\frakP \neq \sigma(\frakP )$ the 
	{\em length} $m$ of a fine admissible lattice chain 
	is just $m=r=\dim _E (V)-1$. 
Also if  $\frakP = \sigma(\frakP) $, then 
$m=r$, where $r$ is the rank of the $\frakp$-adic group defined in
Remark \ref{completiongroup}.
\end{remark}

Note that any admissible chain $\calL$ contains a unique maximal 
integral lattice which we will always denote by $L_0$.

\begin{remark}\label{masschain}
	Let $\calL =(L_0,\ldots , L_r)$ be a fine
	admissible lattice chain for $\frakP $. 
	\begin{itemize}
		\item[(a)] If $\frakP = \sigma(\frakP) $ 
	then $\overline{L_0} := 
	L_0/\frakP L_0^{\#,\frakp }$ is a hermitian space 
	over $\MO/\frakP $ and the spaces
	$V_j:=\frakP L_j ^{\#,\frakp }/ \frakP L_0^{\#,\frakp }$ 
	($j=1,\ldots , r$) define 
	a maximal chain of isotropic subspaces of this hermitian space.
	We call the chain  $(V_1,\ldots , V_{r-1})$ 
	{\em truncated}.
\item[(b)] If $\frakP \neq \sigma(\frakP) $ then 
	$\overline{L_0}:=L_0/\frakP L_0 $ is a vector space 
	over $\MO/\frakP $ and the spaces  $V_j:=L_j/\frakP L_0 $ 
	($j=r,\ldots , 1$) 
	form a maximal chain of subspaces.
	Here we call the chain $(V_{r-1},\ldots , V_1) $
	{\em truncated}.
	\end{itemize}
\end{remark}

For the different hermitian spaces $\overline{L_0}$, the number  of 
such chains of isotropic subspaces 
can be found by recursively applying the formulas in 
\cite[Exercises 8.1, 10.4, 11.3]{Taylor}.

\begin{corollary}\label{einklassig}
	The fine admissible lattice chain $\calL $ represents a one-class genus of lattice chains
	if and only if $L_0$ represents a one-class genus of lattices
	and $\Aut(L_0)$ is transitive on the 
	maximal chains of (isotropic) subspaces of $\overline{L_0}$.
\end{corollary}

\section{Chamber transitive actions on affine buildings.} 

Kantor, Liebler and Tits \cite{KLT} 
classified discrete groups acting chamber transitively and type preservingly 
on the affine building of a simple adjoint algebraic group 
of relative rank $\geq 2$ 
over a locally compact local field. 
Such groups are very rare and hence this situation is
an interesting phenomenon, further studied in 
\cite{Ka1}, \cite{Ka2}, \cite{KMW1},  \cite{MW1}, \cite{KMW90}, 
and \cite{M} (and many more papers by these authors)
where explicit constructions of the groups are given. 
One major disadvantage of the existing literature is that the 
proof in \cite{KLT} is very sketchy, essentially the authors
limit the possibilities that need to be checked to a finite number.

From the classification of the one-class genera of 
admissible fine lattice chains in Section \ref{oneclasschains}, we obtain 
a number theoretic construction of the groups 
in \cite{KLT} over fields of characteristic 0.
It turns out that we find essentially all these 
groups and that our construction allows to find  one more 
case: 
The  building of $\U_5(\Q_3(\sqrt{-3}) )$ 
of type $C-BC_2$, see Proposition \ref{hf} (1),
which, to our best knowledge,  has not appeared in the literature before.

\subsection{S-arithmetic groups} 

We assume that 
$(V,\Phi )$ is a totally positive definite hermitian space, i.e. 
$K$ is totally real and 
$\Phi (x,x) \in K $ is totally positive for all non-zero $x\in V$.

Let $S =\{ \frakp_1,\ldots , \frakp_m \} $ be a finite set
of prime ideals of $\Z_K$.
For a prime ideal $\frakp $ we denote by $\nu _{\frakp}$ the 
$\frakp $-adic valuation of $K$. 
Then the ring of $S$-integers in $K$ is 
$$ \Z_S := 
\{ a\in K \mid \nu _{\frakq} (a) \geq 0 \text{ for all prime ideals } 
\frakq \notin S \} .$$

Let $L$ be some $\MO $-lattice in $(V,\Phi )$ and put
$L_S:= L\otimes_{\Z_K} \Z_S$. 
Then the group 
$$\Aut (L_S ) := \{ g\in \U(V,\Phi ) \mid L_S g = L_S \} $$ 
is an $S$-arithmetic subgroup of $\U(V,\Phi )$. 

\begin{remark} 
For any prime ideal $\frakp$, the group $\U(V,\Phi )$ 
(being a subgroup of $\U_{\frakp }$) 
 acts on the Bruhat-Tits building $\calB$ of the group $\U_{\frakp }$ 
defined in Remark \ref{completiongroup}. 
Assume that the rank of $\U_{\frakp }$ is at least 1. 
The action of the subgroup $\Aut (L_S) $ is discrete and cocompact 
on $\calB$, if and only if $\frakp \in S$ and $(V_{\frakq} ,\Phi )$ is 
anisotropic for all $\frakp\neq \frakq \in S $.
\\
Note that anisotropic spaces only exist up to dimension 4. 
So if we assume that $\U_{\frakp} ^+ $ is simple modulo scalars
and has rank $\geq 2$, then 
$\Aut (L_S) $ is discrete and cocompact 
on $\calB$, if and only if $S=\{ \frakp \}$ 
\end{remark}

\subsection{The action on the building of $\U_{\frakp}$}

In the following we fix a prime ideal $\frakp $ and assume 
that $S=\{ \frakp \} $. 

A lattice class model for the affine building 
$\calB$ has been described in 
\cite{AN}. Note that \cite{AN} imposes the assumption that the 
residue characteristic of $K_{\frakp }$ is $p\neq 2$. 
This is only necessary to obtain a proof of the building axioms 
that is independent from Bruhat-Tits theory. 
For $p=2$, the dissertation \cite{FrischDiss} contains 
the analogous description of the Bruhat-Tits building 
for orthogonal groups. 
For all residue characteristics, 
the chambers in $\calB $ correspond 
to certain fine lattice chains in the natural $\U_{\frakp} $-module $W_{\frakp} $.

Let $L$ be a fixed $\frakp $-normalized lattice in $V$ 
and put $V_{\frakp} := V\otimes _K K_{\frakp} $.

In the case that $E\otimes _K K_{\frakp }$ is a skewfield, 
we decompose the completion 
$$L_{\frakp } = \HH(\MO _\frakp)^r \perp M_0 
= \bigperp _{i=1}^r \langle e_i,f_i \rangle _{\MO _{\frakp}} \perp M_0$$
as in Definition \ref{normalised}.
Then $ V_\frakp =  V_0\perp
\langle e_1,\ldots , e_r,f_1,\ldots , f_r \rangle _{K_{\frakp}} $
where $V_0 = K_{\frakp } M_0$ is anisotropic.
Then the standard chamber corresponding to $L$ and the 
choice of this hyperbolic basis is represented by 
the admissible fine lattice chain 
$$\calL = (L=L_0,L_1,\ldots , L_{r} ) $$
where $L_j \in L(\frakp ) $ is the unique lattice in $V$ such that 
$$(L_j) _{\frakp} = \bigperp _{i=1}^{j} \langle \pi e_i,f_i \rangle _{\MO _{\frakp}} \perp 
\bigperp _{i=j+1}^{r} \langle  e_i,f_i \rangle _{\MO _{\frakp}}
\perp M_0.$$

Now assume that  $E\otimes _K K_{\frakp } \cong K_{\frakp}^{2\times 2}$
and $W_{\frakp } = eV_{\frakp } $ for some primitive idempotent $e$ such that 
$\sigma(e) = 1-e$ as in Remark \ref{completiongroup}. 
Then $W_{\frakp }$ 
carries a symplectic form $\Psi $  and the
 lattice 
 $L_{\frakp }e $ has a symplectic basis $(e_1,f_1,\ldots , e_r,f_r )$, i.e. 
$$L_{\frakp }e = 
 \bigperp _{i=1}^r \langle e_i,f_i \rangle _{\Z _{K_\frakp}} $$
with $\Psi (e_i,f_i ) = 1$. 
The standard chamber corresponding to $L$ and the 
choice of this symplectic basis is represented by 
the admissible fine lattice chain 
$$\calL = (L=L_0,L_1,\ldots , L_{r} ) $$
where $L_j \in L(\frakp ) $ is the unique lattice in $V$ such that 
$$(L_j) _{\frakp} = \bigperp _{i=1}^{j} \langle \pi e_i,f_i \rangle _{\MO _{\frakp}} \perp 
\bigperp _{i=j+1}^{r} \langle  e_i,f_i \rangle _{\MO _{\frakp}}
.$$

In the last and most tricky case 
$E\otimes _K K_{\frakp } \cong K_{\frakp} \oplus K_{\frakp}$. Then $W_{\frakp} = V_{\frakp } e_{\frakP} $ for any of 
the two maximal ideals $\frakP $ of $\MO $ that contain $\frakp$,
$\U_{\frakp } \supseteq \SL(W_{\frakp}) $ and  
$M_{\frakp} := L _{\frakp}  e_{\frakP } $ is a lattice in $W_{\frakp }$. 
To define the standard chamber fix some 
$\Z_{K_\frakp} $-basis $(e_1,\ldots,e_r)$ of $M_{\frakp}$. 
Then the fine admissible lattice chain 
$$\calL = (L=L_0,L_1,\ldots , L_{r} ) $$
where $L_j$ is the unique lattice in $V$ such that 
\begin{itemize} 
	\item 
$(L_j)_{\frakQ} = L_{\frakQ}$  for all prime ideals $ \frakQ \neq \frakP $ of $\MO $
\item
	$(L_j) _{\frakP} = \bigoplus _{i=1}^{j} \langle \pi e_i \rangle _{\MO _{\frakP} } \oplus 
\bigoplus _{i=j+1}^{r} \langle  e_i \rangle _{\MO _{\frakP}} .$
\end{itemize}

	%In this situation let  
	%$G:= \Aut(L_S) $
	%be the $S$-arithmetic group defined by the lattice $L$
	%(where $S=\{ \frakp \} $ as before). 

	\begin{lemma}\label{uniqueunimod} 
		Assume that $\frakP \neq \sigma(\frakP) $, so $E\otimes _K K_{\frakp } \cong K_{\frakp} \oplus K_{\frakp}$ and keep the notation from above. 
	Let $M$ be some $\MO $-lattice in $V$.
	Then
	$
	\{ X \in M(\frakp ) 
	\mid e _{\frakP} X_{\frakp}  =e_{\frakP}M_{\frakp}  \} $
contains  a unique lattice $Y$ 
	with $Y=Y^{\# ,\frakp  }$. 
	%In particular, $\Stab _G(e_{\frakP}M_{\frakp} ) = \Aut (Y) $
	%is a finite group.
\end{lemma}

\begin{proof}
	As $Y\in M(\frakp )$ it is enough to define 
	$Y_{\frakp } = e_{\frakP}M_{\frakp}  \oplus (1-e_{\frakP}) X_{\frakp} $. This $\MO _{\frakp}$-lattice is unimodular if and only if 
	\[ (1-e_{\frakP}) X_{\frakp } = \{ x\in (1-e_{\frakP}) V 
	\mid \Phi ( e_{\frakP}M_{\frakp}, x) \subseteq \MO _{\frakp} \} .\]
\end{proof} 

Thus for $\frakP \neq \sigma(\frakP )$ the stabilizer in the $S$-arithmetic 
group $\Aut(L_S)$ of a vertex in the
building $\calB $ is the automorphism group of a $\frakp $-unimodular lattice. 
Also if $\frakP = \sigma(\frakP )$, 
	any vertex in the building $\calB $ corresponds to
	a unique homothety class of lattices $[M_{\frakp}]
	=\{ a M_{\frakp } \mid a\in K_{\frakp}^* \} $. 
	So by Remark \ref{localglobal} there is a unique 
	lattice  $X\in L(\frakp)$ 
	with $X_{\frakp} = M_{\frakp} $. 
	Hence  the stabilizers of the vertices in $\calB $ are 
	exactly the automorphism groups of the respective lattices in $V$. 
	In particular these are finite groups.

\begin{remark}
	As $\U_{\frakp} $ acts transitively on the chambers 
	of $\calB $, any other chamber (i.e. $r$-dimensional simplex)
	in $\calB$ corresponds to 
	some lattice chain in the genus of $\calL = (L_0,\ldots , L_r) $.
	The $(r-1)$-dimensional simplices are the $\U_{\frakp }$-orbits of 
the subchains 
$\calL _j := ( L_i \mid i\neq j ) $ of $\calL $ for $j=0,\ldots , r$. 
We call these simplices {\em panels} and  $j$ the cotype of the 
panel $\calL _j$. 
\end{remark}

\begin{theorem}\label{main}
	Let $\calL = (L_0,\ldots , L_r)$ be a fine admissible 
	lattice chain for $\frakP $ 
	of class number one. 
	Put $L:=L_0$ and $S:=\{ \frakp \}$. 
	Then $\Aut (L_S )$ acts chamber
	transitively on the (weak) Bruhat-Tits building $\calB$ of the
	completion $\U_{\frakp }$.
\end{theorem} 

\begin{proof}
	We use the characterization of Lemma \ref{einklassig}. 
	Let $\calC $ be the chamber of $\calB $ that corresponds to $\calL $ 
	by the construction above and let $\calD $ be some other chamber
	in $\calB $. Then there is some element $g\in \U_{\frakp }$ with
	$\calC g = \calD $. As the 
	genus of $L$ consists only of one class, there 
	is some $h \in \Aut(L_S) $ such that 
	$g h \in \U_{\frakp }$ stabilizes the vertex $v$ that corresponds to 
	$L$. So $gh \in \Stab _{\U_{\frakp}} (L_{\frakp }) $ 
	and $\calD h $ is some chamber in $\calB $ 
	containing the vertex $v$. Now $\Aut (L) $ acts transitively 
	on the set of all fine admissible lattice chains for $\frakP $ starting in 
	$L$, so there is some $h'\in \Aut(L) $ such that 
	$\calD h h' = \calC $. 
	Thus the element $h h' \in \Aut(L_S) $ maps $\calD $ to $\calC $.
\end{proof}

\subsection{The oriflamme construction} \label{Secoriflamme} 

The buildings $\calB$ described above are in general not thick buildings, 
i.e. there are panels that are only contained in exactly two chambers.
Such panels are called thin. 
To obtain a thick building $\calB ^+$ (with a type preserving action by
the group $\U_{\frakp}^+$ defined in Remark \ref{completiongroup}) 
we need to apply a generalization of the 
oriflamme construction as described in \cite[Section 8]{AN}. 
In particular \cite[Section 8.1]{AN} gives the precise situations 
which panels are thin for the case that $p\neq 2$. 
Also for $p=2$ only the panels of cotype $0$ and $r$ can be thin. 
We refrain from describing the situations for $p=2$ in general, 
but refer to the individual examples below. 

\begin{remark}\label{oriflamme}
	Assume that  $\frakP = \sigma(\frakP )$.  
	\begin{enumerate}
		\item[(a)]
	Assume that there are only two lattices 
	$L_0$ and $L_0'$ in the genus of $L_0$ such
	that 
	$$L_1 \subseteq L_0,L_0' \subseteq L_1^{\#,\frakp }.$$
	Then $\calL $  
	and $\calL' := (L_0', L_1,\ldots , L_{r}) $ are
	the only chambers in $\calB $ that contain the 
	panel $\calL_0 = (L_1,\ldots , L_{r}) $ and hence
	this panel is thin. 
	Then we replace the vertex represented by 
	$L_1$ by the one represented by $L_0'$. 
		\item[(b)]
	Assume that there are only two lattices 
	$L_r$ and $L_r'$ in the genus of $L_r$ such
	that 
	$$\frakP L_{r-1}^{\#,\frakp } \subseteq L_r,L_r' \subseteq L_{r-1}.$$
	Then $\calL $  
	and $\calL' := (L_0, L_1,\ldots , L'_{r}) $ are
	the only chambers in $\calB $ that contain the 
	panel $\calL_r = (L_0,\ldots , L_{r-1}) $ and hence
	this panel is thin. 
	Then we replace the vertex represented by 
	$L_{r-1}$ by the one represented by $L_r'$. 
\item[(c)]
	 After this construction the standard chamber ${\calL}^+ $
	in the 
	thick building $\calB^+$ is either represented by 
	$\calL $, $ (L_0,L_0',L_2,\ldots , L_r )$, 
	$(L_0,L_1,\ldots , L_{r-2},L_r,L_r' )$, or 
	$(L_0,L_0',L_2,\ldots , L_{r-2},L_r,L_r' )$.
	Note that by construction the chain ${\calL }$ can 
	be recovered from ${\calL} ^+ $, so
 the stabilizer of  
 $\calL $ is equal to the stabilizer of 
 all lattices in  $\calL^+$. 
 Moreover every element in $\U_{\frakp }$
 mapping the chain ${\calL }$ to some other chain ${\calL }'$ 
 maps the chamber ${\calL }^+$ to the chamber $(\calL ')^+$ .
 \end{enumerate}
 For more details we refer to \cite[Section 8.3]{AN}. 
\end{remark}

In particular by part (c) of the previous remark we find the 
important corollary.

\begin{corollary}
	In the situation of Theorem \ref{main} the 
	group $\Aut(L_S)$ also acts chamber transitively (not necessarily type preservingly)
	on the thick building $\calB^+$.
\end{corollary}

\begin{remark}\label{splitoriflamme}
	Also in the situation where $\frakP \neq \sigma(\frakP ) $,
 i.e. $E_{\frakp } = K_{\frakp} \oplus K_{\frakp }$,
	the stabilizers of the points in the building are not 
	the stabilizers of the lattices in the lattice chain.
	By Lemma \ref{uniqueunimod} 
	the lattices $L_i$ ($i=1,\ldots ,r$) 
	need to be replaced by the uniquely defined
	lattices $Y_i\in L_i(\frakp )$, such that $(Y_i)_\frakp$ is unimodular 
	(as in Lemma \ref{uniqueunimod}) and  $Y_i\cap L_0 = L_i$. 
	We refer to this construction as a variant of the oriflamme 
	construction in the examples below.
\end{remark}

\begin{theorem}\label{classtwo}
Let $\calL= (L_0 , \dots, L_r)$ be a fine $\frakP$-admissible lattice chain for some maximal two sided 
ideal $\frakP $ of $\MO$ such that $\sigma(\frakP) = \frakP$.
Suppose that the oriflamme construction replaces $\calL$ by some sequence of lattices $\calL^+$ which is one of
$$ \calL, (L_0,L_0',L_2,\ldots , L_r ), 
   (L_0,L_1,\ldots , L_{r-2},L_r,L_r' ) \mbox{ or }
   (L_0,L_0',L_2,\ldots , L_{r-2},L_r,L_r' ).  $$
   Then $L_0$ and $L_0'$ as well as $L_r$ and $L'_r$ are in the same genus but not in the 
   same proper special genus. 
   Put $L:=L_0$ and $S:= \{ \frakp \} $. 
   Then $\Aut ^+(L_S):=\Aut(L_S) \cap \SU (V,\Phi )$ acts type preservingly on the thick building $\calB ^+$.
   This action is chamber transitive if and only if 
    $h^+(L) = 1$ and $\Aut ^+(L) $
    is transitive on the maximal chains (in the first two cases)
    respectively truncated maximal chains  (in the last two cases)
    of isotropic subspaces of $\overline{L}$
defined  in Remark \ref{masschain}.
   %\begin{itemize}
	   %\item[(a)] $\calL ^+ = \calL $, $h^+(L) = 1$ and $\Aut ^+(L) $
%is transitive on the maximal chains of (isotropic) subspaces of $\overline{L}$
			           %defined  in Remark \ref{masschain}.
	   %\item[(b)] $\calL ^+ =  (L_0,L_1,\ldots , L_{r-2},L_r,L_r' )$,
	   %$h^+(L) = 1$ and 
	   %$\Aut ^+(L)$ acts transitive
   %on the 
   %truncated maximal chains of (isotropic) subspaces of $\overline{L}$
         %defined  in Remark \ref{masschain}.
 %\item[(c)] $\calL^+ = (L_0,L_0',L_2,\ldots , L_r ) $, $h^+(L) = 2$ and $\Aut^+(L)$ 
%is transitive on the maximal chains of (isotropic) subspaces of $\overline{L}$
		           %defined  in Remark \ref{masschain}.
   %\item[(d)] $\calL^+ = (L_0,L_0',L_2,\ldots , L_{r-2},L_r,L_r' ) $, 
	   %$h^+(L) = 2$ and $\Aut ^+(L)$ acts transitive on the  
         %truncated maximal chains of (isotropic) subspaces of $\overline{L}$
          %defined  in Remark \ref{masschain}.
	  %\end{itemize}
\end{theorem}

\begin{proof}
	The proof that the action is chamber transitive in all cases 
	is completely analogous to the proof of Theorem \ref{main}. 
	We only need to show that $h^+(L) = 1$.
	So let $M$ be some lattice in the same proper special genus as $L $. 
	By strong approximation for $\U^+(V_A,\Phi )$ (see \cite{KneserApp}), 
	there is some element $g\in \U_{\frakp }^+ $ and $h\in \SU(V,\Phi )$ 
	such that $Mh=Lg$. As $\Aut ^+(L_S) $ is chamber transitive and type preserving, 
	there is some $f\in \Aut ^+(L_S) $ such that $Lf = Mh $ so $M=Lfh^{-1} $ is 
	properly isometric to $L$. 
\end{proof}

To obtain a classification of all chamber transitive discrete actions on $\calB^+$ 
we hence need a classification of all proper spinor genera with proper class number 
one. The thesis \cite{Kirschmer} only lists the genera of class number one and two. 
In some cases, $h(L) = h^+(L) $ for every square-free lattice $L$, for example if:
\begin{itemize}
	\item[(a)] $E=K$, $\dim(V) \geq 5$ and $K$ has narrow class number one 
		(\cite[Theorem 102.9]{OMeara}),
	\item[(b)] $[E:K] = 2$ and $\dim_E(V) $ is odd (\cite{ShimuraUnitary}),
	\item[(c)] or $[E:K] = 4$. 
\end{itemize}

\section{The one-class genera of fine admissible lattice chains}\label{oneclasschains}

We split this section into three subsections dealing with the
different types of hermitian spaces
($[E:K] = 1,2,4 $).
The fourth subsection comments on the exceptional groups.

Suppose $\calL = (L_0,\dots,L_r)$ is  a fine $\frakP$-admissible lattice chain 
of class number one, where $\frakP$ is a maximal two sided ideal of $\MO $.
Then $\frakp := \frakP \cap \Z_K$ together with $L_0$ determines the 
isometry class of $\calL := \calL(L_0,\frakp )$. 
Moreover $L_0$ is a $\frakp $-normalized lattice in $(V,\Phi )$ of class number one
and by Corollary \ref{einklassig} the 
finite group $\Aut(L_0)$ acts transitively on the fine chains of (isotropic) subspaces of $\overline{L_0}$  as in Remark \ref{masschain}.
The one- and two-class genera of lattices in hermitian spaces $(V,\Phi )$
have been  classified in \cite{Kirschmer}.
For all such lattices $L_0$ and all prime ideals $\frakp $, for which $L_0$ is $\frakp$-normalized, 
we check by computer if $\Aut (L_0)$ acts transitively on the
fine chains of (isotropic) subspaces of $\overline{L_0}$.
Note that the number  of such chains grows with the norm of $\frakp$, so the 
order of $\Aut(L_0)$ gives us a bound on the possible prime ideals $\frakp $.
We also checked weaker conditions (similar to the ones in Theorem \ref{classtwo}) that would imply a chamber transitive action on the thick building $\calB ^+$,
i.e. $h(L_0) \leq 2 $ and transitivity only on the truncated
maximal chains. The cases $h(L_0) = 2$ never gave a transitive action
on the chambers of $\calB ^+$.

For any non-empty subset $T$ of $\{1,2,\dots,r\}$ we list the automorphism group $G_T$ of the subchain $(L_i)_{i \in T}$.
With our applications on the action on buildings in mind, we also 
give the order of $$G_T^+:= G_T \cap \U_{\frakp}^+  $$ 
where $\U_{\frakp}^+$ is given in
Remark \ref{completiongroup}.
Note that we will always assume that the rank of the group $\U_{\frakp }$ 
is $r\geq 2$. 

\subsection{Quadratic forms}

In this section suppose that $E=K$.
We denote by $\A_n, \B_n, \D_n, \E_n$ the root lattices of the same type over $\Z_K$.
If $L$ is a lattice and $a\in K$ we denote by 
${}^{(a)} L $ the lattice $L$ with form rescaled by $a$. 
Sometimes we identify lattices over number fields using 
the trace lattice. For instance $(\E_8)_{\sqrt{-3}} $ 
denotes a hermitian lattice over $\Z [ \frac{1+\sqrt{-3}}{2} ]$ of dimension
4 whose trace lattice over $\Z $ is isometric to $\E_8$.

\subsubsection{Quadratic forms in more than four variables} 

If $E=K$, $\dim _K(V) \geq 5$  and $(V,\Phi )$ contains 
a one-class genus of lattices, then by \cite[Section 7.4]{Kirschmer} 
either  $K=\Q $ or $K=\Q[\sqrt{5}]$ where one has essentially 
one one-class genus of lattices of dimension 5 and 6 each. 
The rational lattices have been classified in 
\cite{KL} and are available electronically from \cite{homepage}. 

\begin{prop}\label{qf}
	If $E=K$, $\dim _K(V) \geq 5$  and $(V,\Phi )$ contains 
	a fine $\frakp$-admissible lattice chain $\calL (L_0,\frakp ) $ of class number one for some prime ideal $\frakp $, 
	then $K=\Q $ and $\calL (L_0,\frakp )$ is one of the following 
	nine essentially different chains: 
\begin{enumerate}
	\item $\calL(\E_8, 2) = (\E_8, \D_8, \D_4 \perp \D_4, {}^{(2)}\D_8^\#, {}^{(2)} \E_8)$.
After applying the oriflamme construction, the lattice chain becomes
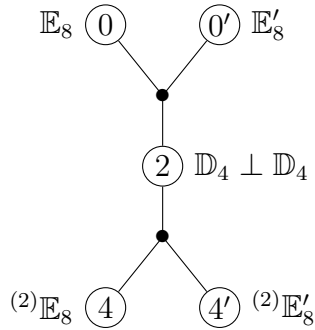
\begin{center}
\begin{tikzpicture}[scale=0.75,-,mynode/.style={circle,draw,inner sep=0pt, minimum size=15pt}]
\node[mynode, label=left:{$\E_8$}] (0) at (-1,2.5) {$0$};
\node[mynode, label=right:{$\E_8'$}] (00) at (1,2.5) {$0'$};
\node (1) at (0,1.25) {$\bullet$};
\node[mynode, label=right:{$\D_4 \perp \D_4$}] (2) at (0,0) {$2$};
\node (3) at (0,-1.25) {$\bullet$};
\node[mynode, label=left:{${}^{(2)} \E_8$}] (4) at (-1,-2.5) {$4$};
\node[mynode, label=right:{${}^{(2)} \E_8'$}] (44) at (1,-2.5) {$4'$};
\draw (0) -- (1.center) -- (2) -- (3.center) -- (4);
\draw (00) -- (1.center);
\draw (3.center) -- (44);
\end{tikzpicture}
\end{center}
The  automorphism groups are as follows
\[
\begin{array}{ccl}
T & G_T & \# G_T^+ \\ \hline
\{i\} & 2.\GO^+_8(2).2 & 2^{13}  \cdot 3^5  \cdot 5^2 \cdot 7 \\
\{2\} & \Aut(\D_4) \wr C_2  & 2^{13} \cdot 3^4 \\
\{i,j\}& 2^{1+6}_+.S_8 &  2^{13} \cdot  3^2  \cdot 5 \cdot 7 \\
\{2,i\}& N.(S_3 \times S_3 \wr C_2) & 2^{13} \cdot 3^3\\
\{i,j,k\} & 2^{1+6}_+.(C_3^2 . \PSL_2(7))  & 2^{13} \cdot 3 \cdot 7\\
\{2,i,j\} & N.(C_2 \times S_3 \wr C_2) & 2^{13} \cdot 3^2\\ 
\{0,0',4,4'\} & 2^{1+6}_+.(C_2^3:S_4)  & 2^{13} \cdot 3 \\
\{2,i,j,k\} & N.(C_2^3 \times S_3) & 2^{13} \cdot 3 \\
\{0,0',2,4,4'\} & N.C_2^3 & 2^{13} 
\end{array}
\]
where $N=O_2(G_{\{2\}}) \isom 2_+^{1+4} \times 2_+^{1+4}$ and $i,j,k \in \{0,0',2,4,4'\}$ with $\#\{i,j,k\} = 3$.

\item $\calL(\E_7, 2) = (\E_7, \D_6 \perp \A_1, \D_4 \perp {}^{(2)}\B_3, {}^{(2)}\B_7)$.
After applying the oriflamme construction, the lattice chain becomes
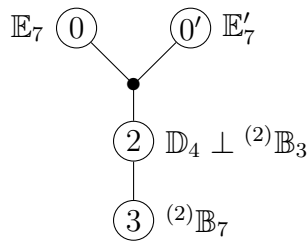
\begin{center}
\begin{tikzpicture}[scale=0.75,-,mynode/.style={circle,draw,inner sep=0pt, minimum size=15pt}]
\node[mynode, label=left:{$\E_7$}] (0) at (-1,2) {$0$};
\node[mynode, label=right:{$\E_7'$}] (00) at (1,2) {$0'$};
\node (1) at (0,1) {$\bullet$};
\node[mynode, label=right:{$\D_4 \perp {}^{(2)} \B_3$}] (2) at (0,0) {$2$};
\node[mynode, label=right:{${}^{(2)}\B_7$}] (3) at (0,-1.4) {$3$};
\draw (0) -- (1.center) -- (2) -- (3);
\draw (00) -- (1.center);
\end{tikzpicture}
\end{center}
The  automorphism groups are as follows
\[
\begin{array}{ccl}
T & G_T & \# G_T^+ \\ \hline
\{i\} & C_2 \times \PSp_6(2) & 2^9 \cdot 3^4 \cdot 5 \cdot 7 \\
\{2\} & \Aut(\D_4) \times C_2 \wr S_3 &  2^{9}  \cdot 3^3  \\
\{3\} & C_2 \wr S_7 & 2^{9} \cdot 3^2 \cdot 5 \cdot 7 \\ 
\{0,0'\} & C_2^6.S_6 & 2^9 \cdot 3^2 \cdot 5 \\  
\{i,2\} & N.S_3^2 & 2^9 \cdot 3^2 \\ 
\{i,3\} & C_2^7.\PSL_2(7) & 2^9 \cdot 3 \cdot 7 \\
\{2,3\} & N.(C_2 \times S_3^2) & 2^{9} \cdot 3^2 \\
\{0,0',2\}, \{i,2,3\} & N.D_{12} & 2^9 \cdot 3 \\
\{0,0',3\} & C_2^7.S_4 & 2^9 \cdot 3 \\ 
\{0,0',2,3\} & N.C_2^2 & 2^9  
\end{array}
\]
where $N:= O_2(G_{\{2\}}) \isom 2^{1+4}_+ \times Q_8$ and $i \in \{0,0'\}$.
The one-class chain 
$$\calL(\B_7, 2) = \{ \B_7, {}^{(2)}(\D_4^\# \perp B_3), {}^{(2)}\D_6^\# \perp \B_1, {}^{(2)} \E_7^\#\}$$ yields the same stabilizers.

\item $\calL(\A_6, 2) = \{ \A_6, X, {}^{(2)} X^{\#,2}, {}^{(2)} \A_6 \}$. Here $X$ is an indecomposable lattice with $\Aut(X) = (C_2^4 \times C_3).D_{12}$. %% $\Aut(X) \isom C_2^2 \times S_3 \times S_4$.
After applying the oriflamme construction, the lattice chain becomes
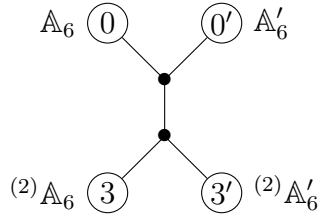
\begin{center}
\begin{tikzpicture}[scale=0.75,-,mynode/.style={circle,draw,inner sep=0pt, minimum size=15pt}]
\node[mynode, label=left:{$\A_6$}] (0) at (-1,2) {$0$};
\node[mynode, label=right:{$\A_6'$}] (00) at (1,2) {$0'$};
\node (1) at (0,1) {$\bullet$};
\node (2) at (0,0) {$\bullet$};
\node[mynode, label=left:{${}^{(2)} \A_6$}] (3) at (-1,-1) {$3$};
\node[mynode, label=right:{${}^{(2)}\A_6'$}] (33) at (1,-1) {$3'$};
\draw (0) -- (1.center) -- (2.center) -- (3);
\draw (00) -- (1.center);
\draw (33) -- (2.center);
\end{tikzpicture}
\end{center}
The  automorphism groups are as follows
\[
\begin{array}{cclc}
	T & G_T & \# G_T^+  & sgdb \\ \hline
	\#T = 1 & C_2 \times S_7 & 2^3 \cdot 3^2 \cdot 5 \cdot 7 & -  \\
	\{0,0'\}, \{3,3'\} & C_2 \times S_3 \times S_4 & 2^3 \cdot 3^2 & 43 \\
	\{0,3\}, \{0,3'\},\{0',3\}, \{0	',3'\} & C_2 \times \PSL_2(7) & 2^3 \cdot 3 \cdot 7  & 42\\
	\# T = 3 & C_2 \times S_4 & 2^3 \cdot 3 & 12\\
	\{0,0', 3,3'\} & C_2 \times D_8 & 2^3 & 3
\end{array}
\]
Here, and in the following tables, the column sgdb gives the label of 
$G_T^+$ 
   as defined by the small group database (\cite{SGDB}).

The admissible one-class chain 
$$\calL({}^{(7)}\A_6^\#,2) = \{ {}^{(7)}\A_6^\#, {}^{(7)} X^{\#, 7}, {}^{(14)} X^\#, {}^{(14)}\A_6^\# \}$$ yields the same groups.

\item $\calL(\E_6, 2) = \{ \E_6, Y_0, \D_4 \perp {}^{(2)}\A_2 \}$.
	Here $Y_0$ is the even sublattice of $\B_5 \perp {}^{(3)} \B_1$. It is indecomposable and  $\Aut(Y_0) = \Aut(\B_5 \perp {}^{(3)} \B_1 ) \isom C_2 \times C_2 \wr S_5$.
After applying the oriflamme construction, the lattice chain becomes
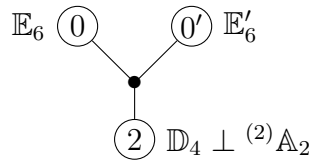
\begin{center}
\begin{tikzpicture}[scale=0.75,-,mynode/.style={circle,draw,inner sep=0pt, minimum size=15pt}]
\node[mynode, label=left:{$\E_6$}] (0) at (-1,2) {$0$};
\node[mynode, label=right:{$\E_6'$}] (00) at (1,2) {$0'$};
\node (1) at (0,1) {$\bullet$};
\node[mynode, label=right:{$\D_4 \perp {}^{(2)} \A_2$}] (2) at (0,0) {$2$};
\draw (0) -- (1.center) -- (2);
\draw (00) -- (1.center);
\end{tikzpicture}
\end{center}
The  automorphism groups are as follows
\[
\begin{array}{cclc}
	T & G_T & \# G_T^+  & sgdb \\ \hline
	\{i\} &  C_2 \times \U_4(2).2 & 2^6 \cdot 3^4 \cdot 5& -  \\
	\{2\} & \Aut(\D_4) \times D_{12} & 2^6 \cdot 3^3 & - \\
	\{0,0'\} & C_2 \wr S_5 & 2^6 \cdot 3 \cdot 5 & 11358 \\ 
	\{i,2\} & N.S_3^2 & 2^6 \cdot 3^2  & 8277 \\
	\{0,0',2\} & N.D_{12} & 2^6 \cdot 3 & 201
\end{array}
\]
where $N=O_2(G_3) \isom 2_+^{1+4} \times C_2$ and $i \in \{0,0'\}$.
The admissible one-class chains 
\begin{align*}
	\calL(\A_2 \perp \D_4, 2) &= \{ \A_2 \perp \D_4, {}^{(2)} Y^{\#,2}, {}^{(2)} \E_6 \} \\ 
	\calL({}^{(3)} (\A_2^\# \perp \D_4) ,2) &= \{ {}^{(3)} (\A_2^\# \perp \D_4), {}^{(6)} Y^\#, {}^{(6)} \E_6 \} \\
	\calL({}^{(3)} \E_6^\#, 2) &= \{ {}^{(3)} \E_6^\#, {}^{(3)} Y^{\#, 3}, {}^{(3)} (\A_2 \perp \D_4)^{\#,3} \}
\end{align*}
yield the same stabilizers.

\item $\calL(\D_6, 2) = \{ \D_6, \D_4 \perp {}^{(2)} \B_2, {}^{(2)}\B_6 \}$.
Here the application of the oriflamme construction is not necessary.
The  automorphism groups are as follows
\[
\begin{array}{cclc}
	T & G_T & \# G_T^+ & sgdb\\ \hline
	\{0\}, \{2\} & C_2 \wr S_7 & 2^8 \cdot 3^2 \cdot 5  & - \\
	\{1\} & \Aut(\D_4) \perp C_2 \wr S_2 & 2^8 \cdot 3^2  & - \\
	\{ 0,1 \}, \{1,2\} & C_2^6.(C_2\times S_4) & 2^8 \cdot 3 & 1086007 \\ 
	\{0,2\} &  C_2^6.(C_2 \times S_4) &  2^8 \cdot 3  & 1088660 \\
	\{0,1,2\} & C_2^6.(C_2 \times D_8) & 2^8    & 6331
\end{array}
\]

\item $\calL(\E_6, 3) = \{ \E_6, \A_2^3, {}^{(3)}\E_6 \}$.
Here the application of the oriflamme construction is not necessary.
The  automorphism groups are as follows
\[
\begin{array}{cclc}
	T & G_T & \# G_T^+ & sgdb \\ \hline
	\{0\}, \{2\} & C_2 \times \U_4(2).2 & 2^6 \cdot 3^4 \cdot 5 & - \\
	\{1\} & D_{12} \wr S_3  & 2^5 \cdot 3^4 & - \\
	\{0,2\} & 3^{1+2}_+.(C_2 \times \GL_2(3)) & 2^3 \cdot 3^4 & 533 \\ 
	\{0,1\}, \{1,2\} & N.(C_2^2 \times S_4) & 2^3 \cdot 3^4  & 704\\ 
	\{0,1,2\} & N.(C_2^2 \times S_3) & 2 \cdot 3^4 & 10 \\
\end{array}
\]
where $N=O_3(G_{\{1\}}) \isom C_3^3$.

\item $\calL(\B_5 \perp {}^{(3)} \B_1, 3) = \{ \B_5 \perp {}^{(3)} \B_1, \B_2 \perp \A_2 \perp {}^{(3)} \B_2; \B_1 \perp {}^{(3)} \B_5 \}$.
Here the application of the oriflamme construction is not necessary.
The  automorphism groups are as follows
\[
\begin{array}{cclc}
	T & G_T & \# G_T^+ & sgdb \\ \hline
	\{0\}, \{2\} & C_2 \times C_2 \wr S_5 & 2^6 \cdot 3 \cdot 5 & - \\
	\{1\} & C_2 \wr S_2 \times D_{12} \times C_2 \wr S_2 & 2^5 \cdot 3 & 144  \\
	\{0,2\} & C_2^2\times \GL_2(3) & 2^3 \cdot 3 & 3  \\
	\{0,1\}, \{1,2\} & C_2^2 \times D_8 \times S_3 & 2^3 \cdot 3 & 8  \\
	\{0,1,2\} & C_2^3 \times S_3 & 2 \cdot 3 & 2
\end{array}
\]
For $0\le i \le 2$ let $Y_i$ be the even sublattice of $\calL(\B_5 \perp {}^{(3)} \B_1, 3)_i$, see also part (4).
Then the admissible one-class chains
$$\calL(Y_0, 3) = \{ Y_0, Y_1, Y_2 \} \mbox{ and } \calL({}^{(2)} Y_0^{\#,2}, 3) = ( {}^{(2)} Y_0^{\#,2}, {}^{(2)} Y_1^{\#,2}, {}^{(2)} Y_2^{\#,2}  )$$
yield the same groups.

\item $\calL(\A_5, 2) = \{\A_5, {}^{(2)} \B_1 \perp Z, {}^{(2)} (\A_2 \perp \B_3)\}$. 
	Here $Z$ is the even sublattice of $\B_3 \perp {}^{(3)} \B_1$ and $\Aut(Z) = \Aut( \B_3 \perp {}^{(3)} \B_1 )$.
After applying the oriflamme construction, the lattice chain becomes
\begin{center}
\begin{tikzpicture}[scale=0.75,-,mynode/.style={circle,draw,inner sep=0pt, minimum size=15pt}]
\node[mynode, label=left:{$\A_5$}] (0) at (-1,2) {$0$};
\node[mynode, label=right:{$\A_5'$}] (00) at (1,2) {$0'$};
\node (1) at (0,1) {$\bullet$};
\node[mynode, label=right:{${}^{(2)} (\A_2 \perp \B_3)$}] (2) at (0,0) {$2$};
\draw (0) -- (1.center) -- (2);
\draw (00) -- (1.center);
\end{tikzpicture}
\end{center}
The  automorphism groups are as follows
\[
\begin{array}{cclc}
	T & G_T & \# G_T^+ & sgdb   \\ \hline
	\{0\}, \{0'\}  & C_2 \times S_6 & 2^3 \cdot 3^2 \cdot 5 & 118 \\ 
	\{2\}  & D_{12} \times C_2 \wr S_3 & 2^3 \cdot 3^2  & 43 \\
	\# T = 2 & C_2^2 \times S_4 & 2^3 \cdot 3 & 12 \\
	\{0,0',2\}  & C_2^2 \times D_8 & 2^3  & 3
\end{array}
\]
The admissible one-class chains
\begin{align*}
	\calL({}^{(3)}\A_5^{\#,3},2) &= \{ {}^{(3)}\A_5^{\#,3}, {}^{(6)}\B_1 \perp {}^{(3)} Z^{\#, 3}, {}^{(6)} (\A_2^\# \perp \B_3) \} \\
	\calL(\A_2 \perp \B_3 ,2) &= \{ \A_2 \perp \B_3 ,  \B_1 \perp {}^{(2)} Z^{\#,2}, {}^{(2)} \A_5^{\#, 2} \} \\
	\calL( {}^{(3)}(\A_2^\# \perp \B_3), 2) &= \{  {}^{(3)}(\A_2^\# \perp \B_3), {}^{(3)}\B_1 \perp {}^{(6)} Z^\#, {}^{(6)} \A_5^\# \}
\end{align*}
yield the same stabilizers.

\item $\calL(\B_5, 3) = \{ \B_5, \B_2 \perp \A_2 \perp {}^{(3)} \B_1, \B_1 \perp {}^{(3)} \B_4 \}$.
After applying the oriflamme construction, the lattice chain becomes
\begin{center}
\begin{tikzpicture}[scale=0.75,-,mynode/.style={circle,draw,inner sep=0pt, minimum size=15pt}]
\node[mynode, label=left:{$\B_5$}] (0) at (-1,2) {$0$};
\node[mynode, label=right:{$\B_5'$}] (00) at (1,2) {$0'$};
\node (1) at (0,1) {$\bullet$};
\node[mynode, label=right:{$(\B_1 \perp {}^{(3)}\B_4)$}] (2) at (0,0) {$2$};
\draw (0) -- (1.center) -- (2);
\draw (00) -- (1.center);
\end{tikzpicture}
\end{center}
\[
\begin{array}{cclc}
	T & G_T & \# G_T^+  & sgdb  \\ \hline
	\{i\}  & C_2 \wr S_5 & 2^6  \cdot 3 \cdot 5  & 11358 \\ 
	\{2\}  & C_2 \wr S_4 \times C_2 & 2^5 \cdot 3 & 204 \\ 
	\{0,0'\} & (C_2 \times D_8) \times S_3 & 2^3 \cdot 3 & 3\\ 
	\{i,2\}  & C_2 \times \GL_2(3) & 2^3 \cdot 3  & 8 \\ 
	\{0,0',2\}  & C_2^2 \times S_3 & 2 \cdot 3  & 2 \\ 
\end{array}
\]
where $i \in \{0,0'\}$. The admissible one-class chain 
\[ \calL( \B_4 \perp {}^{(3)} \B_1 ,3) = \{  \B_4 \perp {}^{(3)} \B_1, \B_1 \perp \A_2 \perp {}^{(3)}  \B_2, {}^{(3)} \B_5 \} \]
yields the same stabilizers.

\end{enumerate}
\end{prop}

\subsubsection{Quadratic forms in four variables} 

Now assume that $K=E$ and $\dim _K(V) = 4$.
By \cite[Theorem 7.4.1]{Kirschmer} there are up to similarity exactly 481 
one-class genera of lattices if $K=\Q$ and additionally 
604 such genera over 21 other base fields where the largest degree is 
$[K:\Q ] =5$ (\cite[Theorem 7.4.2]{Kirschmer}). 
As we are only interested in the case where 
the rank of $\U_{\frakp }$ is $2$, we only need to consider pairs
$(L,\frakp )$  where $L$ is one of these 1095 lattices and $\frakp $ 
a prime ideal such that 
$V_{\frakp} \cong \HH (K_{\frakp} )   \perp \HH (K_{\frakp }) $. 
In this case the building ${\mathcal B}$ of $\U_{\frakp} $ is of type $A_1 \oplus A_1 $ 
and not connected even after oriflamme construction. 
We will not list the groups acting chamber transitively on ${\mathcal B}^+$,
also because of the numerous cases of one-class lattice chains 
in this situation.

To list the lattices we 
 need some more notation.
We denote by $\calQ:= \calQ_{\alpha, \infty, \frakp_1,\dots, \frakp_s}$ a definite quaternion algebra over $K= \Q(\alpha)$ which ramifies exactly at the finite places $\frakp_1,\dots,\frakp_s$ of $K$.
Given an integral ideal $\fraka$ of $\Z_K$ coprime to all $\frakp_i$, then 
$\calO_{\alpha, \infty, \frakp_1,\dots, \frakp_s; \fraka}$ denotes an Eichler order of level $\fraka$ in $\calQ$.

We omit the subscript $\alpha$ whenever $K=\Q$.
Similarly, the subscript $\fraka$ is omitted, if $\fraka=\Z_K$, i.e. the order is maximal.

Then $\calO_{\alpha, \infty, \frakp_1,\dots, \frakp_s; \fraka}$ with the reduced norm form of $\calQ$ yields a quaternary lattice over $\Z_K$.
By \cite[Corollary 4.6]{Nebe} this lattice is unique in its genus, if and only if all Eichler orders of level $\fraka$ in $\calQ$ are conjugate.

Hence we identify such orders with their quaternary lattices.

\begin{prop}
Let $L$ be a $\frakp$-normalized, quaternary lattice over $\Z_K$ such that $\calL(L, \frakp)$ is a 
fine $\frakp$-admissible lattice chain of length $2$ and 
class number one.
Then one of the following holds.
\begin{enumerate}
\item $K=\Q$ and either
\begin{itemize}
\item $\frakp = 2$ and $L \isom \calO_{\infty,3} \isom \A_2 \perp \A_2$ or $\calO_{\infty, 5}$.
\item $\frakp \in \{3,5,11\}$ and $L \isom \calO_{\infty,2} \isom \D_4$.
\item $\frakp = 3$ and $L \isom \B_4$.
\end{itemize}
\item $K = \Q(\sqrt{5})$ and either
\begin{itemize}
\item $\Nr_{K/\Q}(\frakp) \in \{ 4,5,9,11,19,29,59 \}$ and $L \isom \calO_{\sqrt{5},\infty}$. This lattice is called $H_4$ in \cite{Scharlau}.
\item $\Nr_{K/\Q}(\frakp) \in \{ 5, 11\}$ and $L  \isom \calO_{\sqrt{5},\infty; 2 \Z_K} \isom \D_4$.
\item $\frakp = 2\Z_K$ and $L  \isom \calO_{\sqrt{5},\infty; \fraka} \isom\calL( \calO_{\sqrt{5},\infty} , \fraka)_2$ with $\Nr_{K/\Q}(\fraka) \in \{5,11\} $.
\end{itemize}
\item $K = \Q(\sqrt{2})$ and either
\begin{itemize}
\item $\Nr_{K/\Q}(\frakp) \in \{ 2, 7, 23 \}$ and $L \isom \calO_{\sqrt{2},\infty}$.
\item $\Nr_{K/\Q}(\frakp) = 7$ and $L \isom \calO_{\sqrt{2},\infty; \sqrt{2} \Z_K} \isom \calL( \calO_{\sqrt{2},\infty} )_2$ or $L$ is isometric to a unimodular lattice of norm $\sqrt{2}\Z_K$ in $(V,\Phi) \isom \langle 1,1,1,1\rangle$.
  By \cite[IX:93]{OMeara}, the genus of the latter lattice is uniquely determined and it has class number one by \cite{Kirschmer}.
\item $\frakp = \sqrt{2}\Z_K$ and $L \isom \calO_{\sqrt{5},\infty; \fraka} \isom \calL( \calO_{\sqrt{5},\infty}, \fraka )_2$ with $\Nr_{K/\Q}(\fraka) = 7$.
\end{itemize}
\item $K = \Q(\sqrt{3})$ and either
\begin{itemize}
  \item $\frakp = \sqrt{3} \Z_K$ and $L\isom \calO_{\sqrt{3},\infty;\frakp_2}$ or $L$ is isometric to a unimodular lattice of norm $\frakp_2$ in $(V,\Phi) \isom \langle 1,1,1,1\rangle$. Again, this lattice is unique up to isometry.
  \item $\frakp = \frakp_2$ and $L \isom \calO_{\sqrt{3},\infty;\sqrt{3} \Z_K}$.
\end{itemize}
\item $K = \Q(\sqrt{13})$ and $\Nr_{K/\Q}(\frakp) = 3$ and $L \isom \calO_{\sqrt{13},\infty}$. This lattice is called $D_4^\sim$ in \cite{Scharlau}.
\item $K = \Q(\sqrt{17})$ and $\Nr_{K/\Q}(\frakp) = 2$ and $L \isom \calO_{\sqrt{17},\infty}$. This lattice is called $(2A_2)^\sim$ in \cite{Scharlau}.
\item $K = \Q(\theta_9)$ is the maximal totally real subfield of the cyclotomic field $\Q(\zeta_9)$ and $\frakp = 2\Z_K$ and  $L \isom \calO_{\theta,\infty, \frakp_3}$.
\item $K = \Q(\alpha) \isom \Q[X]/(X^3-X^2-3X+1)$ is the unique totally real number field of degree $3$ and discriminant $148$. 
  Then either $\frakp= \frakp_5$ and $L \isom \calO_{\alpha,\infty;\frakp_2}$ or $\frakp= \frakp_2$ and $L \isom \calO_{\alpha,\infty;\frakp_5}$.
\item $K = \Q(\alpha) \isom \Q[X]/(X^3-X^2-4X+2)$ is the unique totally real number field of degree $3$ and discriminant $316$. 
  Then $\frakp=\frakp_2$ and $L \isom \calO_{\alpha, \infty; \frakp_4}$.
\item $K = \Q(\alpha) \isom \Q[X]/(X^4-X^3-3X^2+X+1)$ is the unique totally real number field of degree $4$ and discriminant $725$.
  Then $L \isom \calO_{\alpha, \infty}$ and $\Nr_{K/\Q}(\frakp) \in \{11, 19\}$ or $\frakp$ is the ramified prime ideal of norm $29$.
\item $K = \Q(\alpha) \isom \Q[X]/(X^4-4X^2-X+1)$ is the unique totally real number field of degree $4$ and discriminant $1957$. Then $\frakp = \frakp_3$ and $L \isom \calO_{\alpha, \infty}$.
\item $K = \Q(\alpha) \isom \Q[X]/(X^4-X^3-4X^2+X+2)$ is the unique totally real number field of degree $4$ and discriminant $2777$. 
  Then $\frakp = \frakp_2$ and $L \isom \calO_{\alpha, \infty}$.
\end{enumerate}
Here $\frakp_q$ denotes a prime ideal of $\Z_K$ of norm $q$.
Conversely, in all these cases the chain $\calL(L, \frakp)$ is $\frakp$-admissible and has class number one.
\end{prop}

\subsection{Hermitian forms}

In this section we treat the case that $[E:K]=2$, so 
$E$ is a totally complex extension of degree 2 of the totally 
real number field $K$. All hermitian lattices with class number 
$\leq 2$ are classified in \cite[Section 8]{Kirschmer} and 
listed explicitly for $n\geq 3$ in \cite[pp 129-140]{Kirschmer}. 

\begin{prop}\label{hf}
	Let $\calL(L_0,\frakp)$ be a fine $\frakP$-admissible chain of class number one and of length at least $2$. 
	Then $K  =\Q$, $d:= \dim_E(V) \in \{3,4,5\}$ and one the following holds:
\begin{enumerate}
  \item $E=\Q(\sqrt{-3})$, $\frakp=3\Z$ and $L_0 \isom \B_5 \tensor_\Z \Z[\tfrac{1+\sqrt{-3}}{2}] \isom (\A_2^5)_{\sqrt{-3}}$:
	$$\calL(L_0, 3) = \{ L_0, (\A_2^2 \perp {}^{(3)} \E_6^\# )_{\sqrt{-3}} , (\A_2 \perp \E_8)_{\sqrt{-3}} \}.$$
Here the application of the oriflamme construction is not necessary.
The automorphism groups are as follows:
\[
\begin{array}{ccc}
T & G_T & \# G_T^+  \\ \hline
\{0\} & C_6 \wr S_5 & 2^7 \cdot 3^5 \cdot 5 \\
\{1\} & C_6 \wr S_2 \times {}_{\sqrt{-3}} [ \pm 3^{1+2}_+.\SL_2(3)]_3  & 2^6 \cdot 3^5 \\
\{2\} & C_6 \times {}_{\sqrt{-3}}[\Sp_4(3) \times  C_3]_4 & 2^7 \cdot 3^5 \cdot 5 \\
\{0,1\} & C_6 \wr S_2 \times {}_{\sqrt{-3}} [ \pm 3^{1+2}_+.C_6 ]_3 & 2^4 \cdot 3^5 \\
\{1,2\} & C_6 \times {}_{\sqrt{-3}} [\pm(3_+^{1+2}.\SL_2(3) \times C_3) ]_4 & 2^4 \cdot 3^5 \\ 
\{0,1\} & C_6 \times {}_{\sqrt{-3}} [\pm 3^3 : S_4 \times C_3]_4 & 2^4 \cdot 3^5 \\
\{1,2\} & C_6 \times {}_{\sqrt{-3}} [\pm(3_+^{1+2}.\SL_2(3) \times C_3) ]_4 & 2^4 \cdot 3^5 \\ 
\{0,1,2\}&C_6 \times {}_{\sqrt{-3}} [\pm 3^{1+2}_+.C_6 \times C_3]_4 & 2^2 \cdot 3^5
\end{array}
\]

\item  $E=\Q(\sqrt{-7})$, $\frakp=2\Z$ and $L_0 \isom (\E_8)_{\sqrt{-7}}$:
$$\calL( (\E_8)_{\sqrt{-7}}, 2) = \{ (\E_8)_{\sqrt{-7}} , (\D_8)_{\sqrt{-7}}, (\D_4  \perp \D_4)_{\sqrt{-7}}, ({}^{(2)}\D_8)_{\sqrt{-7}} \}.$$
After applying the variant of the 
oriflamme construction described in Remark \ref{splitoriflamme}, the lattice chain becomes
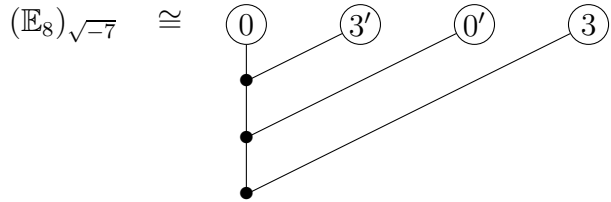
\begin{center}
\begin{tikzpicture}[scale=0.75,-,mynode/.style={circle,draw,inner sep=0pt, minimum size=15pt}]
\node[mynode, label=left:{$(\E_8)_{\sqrt{-7}} \quad \isom \quad $}] (0) at (0,0) {$0$};
\node[mynode] (33) at (2,0) {$3'$};
\node[mynode] (00) at (4,0) {$0'$};
\node[mynode] (x) at (6,0) {$3$};
\node (1) at (0,-1) {$\bullet$};
\node (2) at (0,-2) {$\bullet$};
\node (3) at (0,-3) {$\bullet$};
\draw (0) -- (1.center) -- (2.center) -- (3.center);
\draw (33) -- (1.center);
\draw (00) -- (2.center);
\draw (3.center) -- (x);
\end{tikzpicture}
\end{center}
The  automorphism groups are as follows:
\[
\begin{array}{cclc}
	T & G_T & \# G_T^+ &sgdb \\ \hline
	\#T = 1 & 2.\Alt_7 & 2^4 \cdot 3^2 \cdot 5 \cdot 7 & - \\
	\{0,0'\}, \{3,3'\} & \SL_2(3) \times C_3 :2 & 2^4 \cdot 3^2 & 124 \\
	\{0,3\}, \{0,3'\},\{0',3\}, \{0 ',3'\} & \SL_2(7) & 2^4 \cdot 3 \cdot 7 & 114 \\
	\# T = 3 & 2.S_4 & 2^4 \cdot 3 & 28  \\
	\{0,0', 3,3'\} & Q_{16} & 2^4 & 9 
\end{array}
\]

\item  $E=\Q(\sqrt{-3})$, $\frakp=2\Z$ and $L_0 \isom  (\E_8)_{\sqrt{-3}}$:
	$$\calL( (\E_8)_{\sqrt{-3}}, 2) = \{ (\E_8)_{\sqrt{-3}}, (\D_4 \perp \D_4)_{\sqrt{-3}}, ({}^{(2)}\E_8)_{\sqrt{-3}}\}.$$
Here the application of the oriflamme construction is not necessary.
The automorphism groups are as follows:
\[
\begin{array}{ccl}
T & G_T & \# G_T^+ \\ \hline
\{0\}, \{2\} & \ _{\sqrt{-3}}[\Sp_4(3) \times C_3]_4   & 2^7 \cdot 3^4 \cdot 5 \\
\{1\} & \ _{\sqrt{-3}}[\SL_2(3) \times C_3]_2^2 & 2^7 \cdot 3^3 \\
\{0,2\} & 2^{1+4}_-.\Alt_5 \times C_3 & 2^7 \cdot 3 \cdot 5 \\ 
\{0,1\}, \{1,2\} & \SL_2(3) \wr C_2 \times C_3 & 2^7 \cdot 3^2 \\
\{0,1,2\} & (Q_8 \wr S_2):C_3 \times C_3 & 2^7 \cdot 3
\end{array}
\]

     \item $E = \Q(\sqrt{-1})$,  $\frakp = 2\Z$ and $L_0 \isom (\E_8)_{\sqrt{-1}}$:
	    $$\calL( (\E_8)_{\sqrt{-1}}, 2) = \{ (\E_8)_{\sqrt{-1}}, (\D_8)_{\sqrt{-1}}, (\D_4\perp \D_4) _{\sqrt{-1}} \}.$$
	    Here the application of the oriflamme construction is not necessary.
	     After applying the oriflamme construction, one obtains the following  lattices
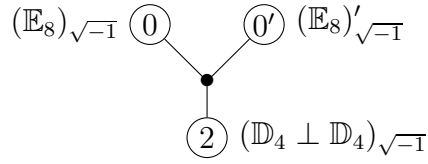
\begin{center}
\begin{tikzpicture}[scale=0.75,-,mynode/.style={circle,draw,inner sep=0pt, minimum size=15pt}]
\node[mynode, label=left:{$(\E_8)_{\sqrt{-1}}$}] (0) at (-1,2) {$0$};
\node[mynode, label=right:{$(\E_8)_{\sqrt{-1}}'$}] (00) at (1,2) {$0'$};
\node (1) at (0,1) {$\bullet$};
\node[mynode, label=right:{$(\D_4\perp \D_4) _{\sqrt{-1}}$}] (2) at (0,0) {$2$};
\draw (0) -- (1.center) -- (2);
\draw (00) -- (1.center);
\end{tikzpicture}
\end{center}
	     \[
		     \begin{array}{ccl}
			     T & G_T & \# G_T^+ \\ \hline
			     \{0\}, \{0'\} & {}_i[(2^{1+4}_+ {\textsc Y} C_4).S_6]_4  & 2^9 \cdot 3^2 \cdot 5 \\
			     \{2\} & \ _{i}[(D_8 {\textsc Y} C_4 ).S_3]_2^2 & 2^9 \cdot 3^2 \\
			     \{ 0,2 \}, \{0',2\} &  & 2^9 \cdot 3 \\ 
			     \{0,0'\} &   &  2^9 \cdot 3 \\
			     \{0,2,0'\} & & 2^9   
		     \end{array}
	     \]

\item $E = \Q(\sqrt{-3})$, $\frakp = 3\Z$ and $L_0 = \B_4 \tensor_\Z \Z[\frac{1+\sqrt{3}}{1}] \isom (\E_8)_{\sqrt{-3}}$:
              Here the application of the oriflamme construction is not necessary.
$$\calL( L_0, 3) = \{ (\A_2^4)_{\sqrt{-3}} , (\A_2 \perp {}^{(3)}\E_6^\#)_{\sqrt{-3}}, ({}^{(3)}\E_8)_{\sqrt{-3}} \}.$$
After applying the oriflamme construction, the chain becomes:
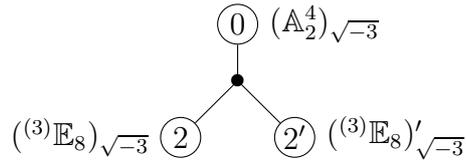
\begin{center}
\begin{tikzpicture}[scale=0.75,-,mynode/.style={circle,draw,inner sep=0pt, minimum size=15pt}]
\node[mynode, label=left:{$({}^{(3)}\E_8)_{\sqrt{-3}}$}] (2) at (-1,-2) {$2$};
\node[mynode, label=right:{$({}^{(3)}\E_8)_{\sqrt{-3}}'$}] (22) at (1,-2) {$2'$};
\node (1) at (0,-1) {$\bullet$};
\node[mynode, label=right:{$(\A_2^4) _{\sqrt{-3}}$}] (0) at (0,0) {$0$};
\draw (0) -- (1.center) -- (2);
\draw (22) -- (1.center);
\end{tikzpicture}
\end{center}
The  automorphism groups are as follows:
\[
\begin{array}{cclc}
	T & G_T & \# G_T^+ & sgdb \\ \hline
	\{0\} & C_6 \wr S_4  & 2^6 \cdot 3^4  & - \\
	\{2\}, \{2'\} & \ _{\sqrt{-3}}[\Sp_4(3) \times C_3]_4   & 2^7 \cdot 3^4 \cdot 5  & - \\
	\{0,2\}, \{0,2'\} & (\pm C_3^4).S_4  & 2^4 \cdot 3^4 & 3085 \\ 
	\{2,2'\}& (C_6 \times 3_+^{1+2}).S_3   & 2^4 \cdot 3^4 & 2895 \\
	\{0,2,2'\} & (C_6 \times C_3 \wr C_3).2 & 2^2 \cdot 3^4 & 68 
\end{array}
\]

\item
$E=\Q(\sqrt{-7})$, $\frakp=2\Z$ and $L_0 = ({}^{(7)} \A_6^\#)_{\sqrt{-7}}$.
After applying the variant of the oriflamme construction described in 
Remark \ref{splitoriflamme}, the chain becomes:
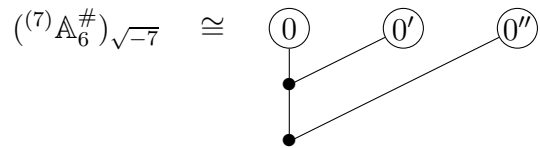
\begin{center}
\begin{tikzpicture}[scale=0.75,-,mynode/.style={circle,draw,inner sep=0pt, minimum size=15pt}]
\node[mynode, label=left:{$ ({}^{(7)} \A_6^\#)_{\sqrt{-7}}  \quad \isom \quad $}] (0) at (0,0) {$0$};
\node[mynode] (11) at (2,0) {$0'$};
\node[mynode] (00) at (4,0) {$0''$};
\node (1) at (0,-1) {$\bullet$};
\node (2) at (0,-2) {$\bullet$};
\draw (0) -- (1.center) -- (2.center);
\draw (11) -- (1.center);
\draw (00) -- (2.center);
\end{tikzpicture}
\end{center}
The  automorphism groups are as follows:
\[
\begin{array}{ccl}
T & G_T & \# G_T^+ \\ \hline
\# T = 1 & \pm C_7 : 3 &  3 \cdot 7 \\ 
\# T = 2 & C_6 & 3 \\
\{0,0',0''\} & C_2 & 1
\end{array}
\]
\end{enumerate}
\end{prop}

\subsection{Quaternionic hermitian forms}

In this section we treat the case that $[E:K]=4$, so 
$E$ is a totally definite quaternion algebra over  the totally 
real number field $K$. All quaternionic hermitian lattices with class number 
$\leq 2$ are classified in \cite[Section 9]{Kirschmer} and 
listed explicitly for $n\geq 2$ in \cite[pp 147-150]{Kirschmer}.

\begin{prop}\label{qaf}
Suppose $E$ is a definite quaternion algebra and let $\calL(L_0,\frakp)$
be a fine $\frakP$-admissible chain of length at least 2 and of class number one.
Then $K = \Q$, $d:= \dim_E(V) = 2$ and one of the following holds:
\begin{enumerate}
\item $E \isom \mathcal{Q}_{\infty, 2}$, the rational quaternion algebra ramified at $2$ and $\infty$, $\frakp=3\Z$ and $L_0 \isom (\E_8)_{\infty, 2}$ is
	the unique $\MO$-structure of the $\E_8$-lattice whose automorphism group is called $ _{\infty,2}[ 2^{1+4}_-.\Alt_5]_2$ in \cite{Nebe}.
The oriflamme construction is not necessary and 
the automorphism groups are
\[
\begin{array}{cclc}
T & G_T & \# G_T & sgdb \\ \hline
	\{0\}, \{2\} &  _{\infty,2}[ 2^{1+4}_-.\Alt_5]_2 & 2^7 \cdot 3 \cdot 5 & - \\
\{1\} & Q_8 : \SL_2(3) & 2^6 \cdot 3 & 1022\\
\{0,1\}, \{1,2\} & C_2 \times \SL_2(3)  & 2^4 \cdot 3 & 32 \\
\{0,2\} & C_3:SD_{16}  & 2^4 \cdot 3 & 16 \\
\# T = 3 & C_2 \times C_6 & 2^2 \cdot 3 & 9
\end{array}
\]

\item $E \isom \mathcal{Q}_{\infty, 3}$ and $\frakp=2\Z$ and $L_0 \isom (\E_8)_{\infty,3}$ is the unique $\MO$-structure of the 
	$\E_8$-lattice whose automorphism group is called 
	$ _{\infty,3}[ \SL_2(9) ]_2$ in \cite{Nebe}.
The oriflamme construction is not necessary and
the automorphism groups are
\[
\begin{array}{cclc}
T & G_T & \# G_T & sgdb \\ \hline
\{0\}, \{2\} & _{\infty,3}[ \SL_2(9) ]_2 & 2^4 \cdot 3^2 \cdot 5 & 409 \\ 
\{1\} & \SL_2(3).S_3 & 2^4 \cdot 3^2 & 124 \\
\#T = 2 & C_2.S_4 & 2^4 \cdot 3 & 28 \\
\{0,1,2\} & Q_{16} & 2^4 & 9   
\end{array}
\]

\end{enumerate}
Note that the above quaternion algebras only have one conjugacy class of maximal orders and for any such order $\MO$, the above $\MO$-lattice $L_0$ is uniquely determined up to isometry.
\end{prop}

%\subsection{A final remark}
%
%In all the cases in Propositions \ref{qf}, \ref{hf} and \ref{qaf}, the 

\subsection{The exceptional groups} \label{g2}

The exceptional groups have been dealt with in \cite[Chapter 10]{Kirschmer},
where it is shown that only the group $\textnormal{G}_2$ admits one-class genera
defined by a coherent family of parahoric subgroups. 
In all cases the number field is the field of rational numbers.
The one-class genera of lattice chains correspond 
to the coherent families of parahoric subgroups $(P_q ) _{q \text{ prime}} $
where for one prime $p$ the 
 parahoric subgroup $P_p$ is the Iwahori subgroup,
a stabilizer of a chamber in the corresponding $p$-adic building. 
Hence \cite[Theorem 10.3.1]{Kirschmer} 
 shows directly that there is 
 a unique $S$-arithmetic group of type $\textnormal{G}_2$ with a discrete and
 chamber transitive  action. 
 It is given by the $\Z $-form ${\bf G}_2$ where each 
 parahoric subgroup $P_q$ is hyperspecial. 
 This integral model of $\textnormal{G}_2$ is described 
 in \cite{Gross} (see also \cite{CNP} for more one-class genera of $\textnormal{G}_2$). 
 Here ${\bf G}_2(\Z ) \cong \textnormal{G}_2(2) $ and the
 $S$-arithmetic group is ${\bf G}_2(\Z [\frac{1}{2}] )$ (so $S=\{  2 \} $). 
 The extended Dynkin diagram of $\textnormal{G}_2$ is as follows.
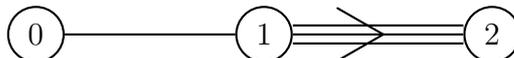
\begin{center}
	%\begin{tikzpicture}[->,>=stealth',shorten >=1pt,auto,node distance=3cm,
	%  thick,main node/.style={circle,fill=blue!20,draw,font=\sffamily\Large\bfseries}]
	\begin{tikzpicture}[-,node distance=3cm,thick,mynode/.style={circle,draw},scale = .75]
		\node[mynode] (1) {0};
		\node[mynode] (2) [right of = 1] {1};
		\node[mynode] (3) [right of = 2] {2};
		\draw[style={postaction={decorate},decoration={markings,mark=at position 0.65 with {\arrow{angle 60}}},double distance=6pt}]
		  (2) -- (3);
		  \draw (1) -- (2) -- (3);
	  \end{tikzpicture}
  \end{center}
  The stabilizers $G_T$  of the simplices $T \subseteq \{ 0,1,2 \}$ 
  in the corresponding 
  building of $\textnormal{G}_2(\Q_2)$ are given in \cite[Section 10.3]{Kirschmer}: 
  %By \cite[3.5.2]{Tits}, the parahoric subgroups $P_q$ of $G_2(\Q_q)$ are in one-to-one correspondence with the non-empty subsets of $\{0,1,2\}$.
  %For any non-empty subset $T$ of $\{0,1,2\}$ let $P_q^T$ be the parahoric subgroup of $G_2(\Q_q)$ whose Dynkin diagram is obtained 
  %from the extended Dynkin diagram of $G_2$ by omitting the vertices in $T$.
  %Then $P_q^{\{0\}}$ is the hyperspecial parahoric subgroup.

  \begin{table}[H]
	  $$
	  \begin{array}{cclc}
		  T &   G_T& \# G_T & sgdb\\ \hline
		  \{0\}  &  \textnormal{G}_2(2) & 2^6 \cdot 3^3 \cdot 7 &- \\
		  \{2\}    &  2^3.\GL_3(2) & 2^6 \cdot 3 \cdot 7 &814 \\
		  \{1\}    & 2^{1+4}_+.((C_3 \times C_3).2) & 2^6 \cdot 3^2 & 8282 \\
		  \{1,2\}  &  2_+^{1+4}.S_3 & 2^6 \cdot 3    &1494 \\
		  \{0,2\}  &  ((C_4 \times C_4).2).S_3 & 2^6 \cdot 3    &956 \\
		  \{0,1\}  &  2_+^{1+4}.S_3 & 2^6 \cdot 3    &988 \\
		  \{0,1,2\}&  \Syl_2(\textnormal{G}_2(2)) &2^6     & 134 \\
	  \end{array}
	  $$
  \end{table}

  One may visualize the chamber transitive action of ${\bf G}_2(\Z[\frac{1}{2}])$ on the Bruhat-Tits building of $\textnormal{G}_2(\Q_2)$ by indicating the 
  three generators $x,y,z$ of ${\bf G}_2(\Z[\frac{1}{2}])$ of order 3 mapping the
  standard chamber to one of the (three times) two neighbors. 

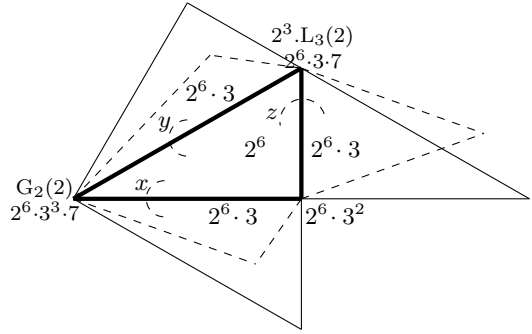
\def\s{1.73205080756887729352744634}
\def\scale{3}
\begin{center}
\begin{tikzpicture}
 \coordinate (A) at (0,0);
 \coordinate (B) at (\scale,0);
 \coordinate (C) at (2*\scale,0);
 \coordinate (D) at (\scale,1/3*\s*\scale);
 \coordinate (E) at (\scale,-1/3*\s*\scale);
 \coordinate (F) at (0.5*\scale,0.5*\scale*\s);
 \coordinate (G) at (1.8*\scale,0.5*1/3*\s*\scale);
 \coordinate (H) at (0.8*\scale,-1/3*\s*0.5*\scale);
 \coordinate (I) at (0.6*\scale,1/3*\s*1.1*\scale);

 \node at (-1/8*\scale,0) {$\substack{\textnormal{G}_2(2) \\ 2^6\cdot3^3\cdot 7}$};
 \node at (\scale+0.05*\scale,1/3*\scale*\s+0.1*\scale) {$\substack{2^3.\mathrm{L}_3(2) \\ 2^6\cdot3\cdot7}$};
 \node at (\scale+0.15*\scale,-0.27*1/4*\scale) {{\scriptsize $2^6\! \cdot 3^2$}};
 \node at (0.8*\scale,0.4*1/3*\s*\scale) {{\scriptsize $2^6$}};
 \node at (\scale+0.15*\scale,0.4*1/3*\s*\scale) {{\scriptsize $2^6\! \cdot 3$}};
 \node at (0.7*\scale,-0.27*1/4*\scale) {{\scriptsize $ 2^6\! \cdot 3$}};
 \node at (0.6*\scale,0.63*1/3*\scale*\s+0.45*1/4*\scale) {{\scriptsize $ 2^6\! \cdot 3$}};
 \node at (0.87*\scale,0.65*1/3*\scale*\s) {\scriptsize $z$};
 \node at (0.3*\scale,0.06*\scale) {\scriptsize $x$};
 \node at (0.4*\scale,0.5*1/3*\scale*\s+0.04*\scale) {\scriptsize $y$};

 \draw [ultra thick] (A) to (B);
 \draw (B) to (C);
 \draw [ultra thick] (B) to (D);
 \draw (D) to (E);
 \draw [ultra thick] (A) to (D);
 \draw (E) to (A);
 \draw (D) to (C);
 \draw (F) to (A);
 \draw (F) to (D);
 \draw [dashed] (B) to (G);
 \draw [dashed] (D) to (G);
 \draw [dashed] (A) to (H);
 \draw [dashed] (B) to (H);
 \draw [dashed] (A) to (I);
 \draw [dashed] (D) to (I);

 \draw [dashed] (1.1*\scale,0.65*1/3*\scale*\s) arc[x radius=0.1*\scale, y radius=0.1*\scale, start angle=20, end angle=190];
 \draw [dashed] (0.4*\scale,0.08*\scale) arc[x radius=0.08*\scale, y radius=0.08*\scale, start angle=90, end angle=270];
 \draw [dashed] (0.5*\scale,0.5*1/3*\scale*\s+0.06*\scale) arc[x radius=0.08*\scale, y radius=0.08*\scale, start angle=90, end angle=270];
\end{tikzpicture}
\end{center}

Using a suitable embedding $\textnormal{G}_2 \hookrightarrow \GO_7$ we find matrices for the three
generators 

\small
$$
x:=\left(\begin{array}{@{}r@{}r@{}r@{}r@{}r@{}r@{}r@{}}
		0 & 1& 1&-1&-1&-1&\phantom{-}0 \\
     0 & 0& 0& 0& 0& 0& 1 \\
     0& 1&-1& 0& 0&-1& 1 \\
     1& 1& 0&-1& 0&-1& 0 \\
     0& 0&-1& 0& 0& 0& 1 \\
     0&-1& 1& 0& 0& 0& 0 \\
     0& 1& 0& 0&-1&-1& 0
\end{array} \right) ,\ 
y:=\left(\begin{array}{@{}r@{}r@{}r@{}r@{}r@{}r@{}r@{}}
		1&\phantom{-}1& 0&-1&-1&-1& 0 \\
     1& 1&-1&-1& 0&-1& 0 \\
     1& 1&-1& 0& 0&-1& 0 \\
     1& 0&-1& 0& 0&-1& 0 \\
     1& 1& 0&-1& 0&-1&-1 \\
     0& 1& 0&-1&-1& 0& 0 \\
     1& 0& 0& 0& 0& 0& 0
\end{array} \right) ,\ 
z:=\frac{1}{2}\left(\begin{array}{@{}r@{}r@{}r@{}r@{}r@{}r@{}r@{}}
		2& 2&\phantom{-}0&-1& 0&-2&-1 \\
 1& 0& 2& 0&-1& 0&-1 \\
 2& 2& 2&-3&-2&-2&-1 \\
 2& 2& 0&-1&-2&-2&-1 \\
 0& 0& 4&-2&-2& 0&-2 \\
 2&-2& 0& 1& 2& 0&-1 \\
 0& 2& 0&-1& 0&-2& 1
\end{array} \right) . $$

\normalsize

To obtain a presentation in these generators, one only needs to 
compute the relations between the pairs of generators that hold in the finite
group generated by the two matrices (in the stabilizer of a vertex).

\section{Chamber transitive actions on $p$-adic buildings.}\label{buildtable}

In this section we tabulate the chamber transitive actions 
on the $p$-adic buildings obtained from the 
one-class genera of lattice chains given in the previous section. 

We use the names and the local Dynkin diagrams as given in \cite{Tits}.
The name for $\U_{\frakp} $ usually does not give the 
precise type of the $p$-adic group. 
For instance the lattices $\E_6$ and $\D_6 $ define two non isomorphic 
non-split forms of the algebraic
group $\GO_6$ over $\Q _2$ which we both denote by $\GO_6^-(\Q_2)$ to avoid clumsy
notation. Computing the discriminant of the invariant form (which is $-3$ respectively
$-1$) we see that for $\E_6$ the orthogonal group splits over the 
unramified extension, whereas for $\D_6$ the group splits over 
the  ramified extension $\Q_2(\sqrt{-1})$ of $\Q_2$. 
Note that the isomorphism $\GO_6^- \cong \U_4$ is given by the action of 
$\GO_6$ on the even part of the Clifford algebra. So we find the 
one-class genera of lattice chains also in a hermitian geometry,
for $\calL (\E _6,2 )$ (from \ref{qf} (3)) we get the same stabilizers as for 
$\calL ((\E_8 ) _{\sqrt{-3}},2 ) $  (from \ref{hf} (2)) in the 
projective group. Such coincidences 
are  indicated by listing the lattices $L_0$ and the 
corresponding references (ref) in Table \ref{table1}. 
The last column of Table \ref{table1} 
refers to a construction of the respective chamber transitive
action in the literature.
For a more detailed description 
of the different unitary groups $\U_{p }$ associated
to the various types of local Dynkin diagrams we refer the reader to 
\cite[Section 4.4]{Tits}.

\begin{table}\caption{Buildings with chamber transitive discrete actions} \label{table1}
\begin{longtable}{|B{0.3}|B{0.3}|B{2.1}|B{1.2}|B{2.5}|B{1.8}|B{3}|B{1}|}\hline
	$p$ & $r$ &  $L_0$ & ref & $\U_{p} $  & name &  local Dynkin & Lit  \\ \hline
	2 & 4 & $\E_8 $ & \ref{qf} (1) & $\GO_8^+(\Q_2) $ & $\tilde{D}_4$ &  
\begin{tikzpicture}[scale=0.75]
\node  (1) at (1.5,0.5) {$0$};
\node  (2) at (1.5,-0.5) {$0'$};
\node  (3) at (0,0) {$2$};
\node  (4) at (-1.5,0.5) {$4$};
\node  (5) at (-1.5,-0.5) {$4'$};
\draw  (1) -- (3) -- (2) -- (3);
\draw  (4) -- (3) -- (5) -- (3);
\end{tikzpicture} & \cite{KLT} \cite{Ka1}
	\\ \hline
	2 & 3 & $\E_7 $ & \ref{qf} (2) & $\GO_7(\Q_2) $ & $\tilde{B}_3$ &  
\begin{tikzpicture}[scale=0.75]
\node  (1) at (1.5,0.5) {$0$};
\node  (2) at (1.5,-0.5) {$0'$};
\node  (3) at (0,0) {$2$};
\node  (4) at (-1.5,0) {$3$};
\draw  (1) -- (3) -- (2) -- (3);
\draw[double] (3) --node{$<$} (4);
\end{tikzpicture} & \cite{KLT} \cite{Ka1} 
\\ \hline
2 & 3 & $\A_6 $, \ $(\E _8)_{\sqrt{-7}}$  & \ref{qf} (3) \ref{hf} (2) & $\GO_6^+(\Q_2) \cong \SL_4(\Q_2) $ & $\tilde{A}_3$ &  
\begin{tikzpicture}[scale=0.75]
\node  (1) at (0,0) {$0$};
\node  (2) at (2,0) {$3$};
\node  (3) at (1,-1) {$0'$};
\node  (4) at (-1,-1) {$3'$};
\draw  (1) -- (2) -- (3) -- (4) -- (1);
\end{tikzpicture} & \cite{KLT} \cite{Ka2}  \\ \hline
2 & 2 & $\E_6 $, \ $(\E _8)_{\sqrt{-3}}$  & \ref{qf} (4) \ref{hf} (3) & $\GO_6^-(\Q_2) \cong \U_4(\Q_2(\sqrt{-3})) $ & $B-C_2$ &  
\begin{tikzpicture}[scale=0.75]
\node  (1) at (-1.5,0) {$0$};
\node  (3) at (0,0) {$2$};
\node  (2) at (1.5,0) {$0'$};
\draw[double] (1) --node{$>$} (3) --node{$<$} (2);
\end{tikzpicture} & \cite{KLT} \cite{MW1} 
\\ \hline
2 & 2 & $\D_6 $, \ $(\E _8)_{\sqrt{-1}}$  & \ref{qf} (5) \ref{hf} (4) & $\GO_6^-(\Q_2) \cong \U_4(\Q_2(\sqrt{-1})) $ & $C-B_2$ &  
\begin{tikzpicture}[scale=0.75]
\node  (1) at (-1.5,0) {$0$};
\node  (3) at (0,0) {$2$};
\node  (2) at (1.5,0) {$0'$};
\draw[double] (1) --node{$<$} (3) --node{$>$} (2);
\end{tikzpicture} & \cite{KLT} \cite{Ka1}
\\ \hline
3 & 2 & $\E_6 $, $\B_5\perp \ ^{(3)} \B_1$, $(\E _8)_{\sqrt{-3}}$  & \ref{qf} (6) \ref{qf} (7)  \ref{hf} (5) & $\GO_6^-(\Q_3) = \GO_6^-(\Q_3) \cong \U_4(\Q_3(\sqrt{-3})) $ & $C-B_2$ &  
\begin{tikzpicture}[scale=0.75]
\node  (1) at (-1.5,0) {$0$};
\node  (3) at (0,0) {$2$};
\node  (2) at (1.5,0) {$0'$};
\draw[double] (1) --node{$<$} (3) --node{$>$} (2);
\end{tikzpicture} & \cite{KLT} \cite{KMW90} 
\\ \hline
2 & 2 & $\A_5 $,  \ \ $(\E _8)_{\infty, 3}$  & \ref{qf} (8) \ref{qaf} (1) & $\GO_5(\Q_2) \cong \Sp_4(\Q_2) $ & $\tilde{C}_2$ &  
\begin{tikzpicture}[scale=0.75]
\node  (1) at (-1.5,0) {$0$};
\node  (3) at (0,0) {$2$};
\node  (2) at (1.5,0) {$0'$};
\draw[double] (1) --node{$>$} (3) --node{$<$} (2);
\end{tikzpicture} & \cite{KLT} \cite{MW1} 
\\ \hline
3 & 2 & $\B_5 $,  \ \  $(\E _8)_{\infty, 2}$  & \ref{qf} (9) \ref{qaf} (2) & $\GO_5(\Q_3) \cong \Sp_4(\Q_3) $ & $\tilde{C}_2$ &  
\begin{tikzpicture}[scale=0.75]
\node  (1) at (-1.5,0) {$0$};
\node  (3) at (0,0) {$2$};
\node  (2) at (1.5,0) {$0'$};
\draw[double] (1) --node{$>$} (3) --node{$<$} (2);
\end{tikzpicture} & \cite{KLT} \cite{KMW90} 
\\ \hline
2 & 2 & $ (\ ^{(7)} \A_6^{\#})_{\sqrt{-7}}  $ 
& \ref{hf} (6) & $\SL_3(\Q_2) $ & $\tilde{A}_2$ &  
\begin{tikzpicture}[scale=0.75]
\node  (1) at (0,0) {$0$};
\node  (3) at (-1,-1) {$0''$};
\node  (2) at (1,-1) {$0'$};
\draw (1) -- (3) -- (2) -- (1);
\end{tikzpicture}
& \cite{KLT}  \cite{KMW1}  \cite{Mu} 
\\ \hline
3 & 2 & $(\A_2^5)_{\sqrt{-3}}$  & \ref{hf} (1) & $\U_5(\Q_3(\sqrt{-3})) $ & $C-BC_2$ &  
\begin{tikzpicture}[scale=0.75]
\node  (1) at (-1.5,0) {$0$};
\node  (3) at (0,0) {$2$};
\node  (2) at (1.5,0) {$0'$};
\draw[double] (1) --node{$<$} (3) --node{$<$} (2);
\end{tikzpicture} & new 
\\ \hline
2 & 2 & ${\bf G}_2(\Z[\frac{1}{2}])$  & \ref{g2}  & $\textnormal{G}_2(\Q_2) $ & $\tilde{\textnormal{G}}_2$ &  
\begin{tikzpicture}[scale=0.75]
\node  (0) at (-1.5,0) {$0$};
\node  (1) at (0,0) {$1$};
\node  (2) at (1.5,0) {$2$};
\draw (0) -- (1);
\draw[double, double distance=3pt] (1) -- node{$>$} (2);
\draw (2) -- (1);
\end{tikzpicture} & \cite{KLT} \cite{Ka1} 
\\ \hline
\end{longtable}
\end{table}

\bibliography{spp}

\begin{thebibliography}{10}

\bibitem{AN}
P.~Abramenko and G.~Nebe.
\newblock Lattice chain models for affine buildings of classical type.
\newblock {\em Math. Ann.}, 322(3):537--562, 2002.

\bibitem{SGDB}
H.~U. Besche, B.~Eick, and E.~A. O'Brien.
\newblock The groups of order at most 2000.
\newblock {\em Electron. Res. Announc. Amer. Math. Soc.}, 7:1--4, 2001.

\bibitem{Borel_finite}
A.~Borel.
\newblock Some finiteness properties of adele groups over number fields.
\newblock {\em Publ. Math. I.H.E.S.}, 16:5--30, 1963.

\bibitem{CNP}
A.~M. Cohen, G.~Nebe, and W.~Plesken.
\newblock Maximal integral forms of the algebraic group {$G_2$} defined by
  finite subgroups.
\newblock {\em J. Number Theory}, 72(2):282--308, 1998.

\bibitem{FrischDiss}
W.~{Frisch}.
\newblock {\em {The cohomology of $S$-arithmetic spin groups and related
  Bruhat-Tits buildings.}}
\newblock G\"ottingen: Univ. G\"ottingen, Mathematisch-Naturwissenschaftliche
  Fakult\"at, 2002.

\bibitem{Gerstein}
L.~Gerstein.
\newblock The growth of class numbers of quadratic forms.
\newblock {\em Amer. J. Math.}, 94(1):221--236, 1972.

\bibitem{Gross}
B.~H. Gross.
\newblock Groups over {$\Z$}.
\newblock {\em Invent. Math.}, 124(1-3):263--279, 1996.

\bibitem{Ka1}
W.~M. Kantor.
\newblock Some exceptional {$2$}-adic buildings.
\newblock {\em J. Algebra}, 92(1):208--223, 1985.

\bibitem{Ka2}
W.~M. Kantor.
\newblock Some locally finite flag-transitive buildings.
\newblock {\em European J. Combin.}, 8(4):429--436, 1987.

\bibitem{KLT}
W.~M. Kantor, R.~A. Liebler, and J.~Tits.
\newblock On discrete chamber-transitive automorphism groups of affine
  buildings.
\newblock {\em Bull. Amer. Math. Soc. (N.S.)}, 16(1):129--133, 1987.

\bibitem{KMW90}
W.~M. Kantor, T.~Meixner, and M.~Wester.
\newblock Two exceptional {$3$}-adic affine buildings.
\newblock {\em Geom. Dedicata}, 33(1):1--11, 1990.

\bibitem{homepage}
M.~Kirschmer.
\newblock Genera of quadratic and hermitian lattices with small class number.
\newblock http://www.math.rwth-aachen.de/homes/Markus.Kirschmer/forms/.

\bibitem{Kirschmer}
M.~Kirschmer.
\newblock {\em Definite quadratic and hermitian form with small class number}.
\newblock Habilitation, RWTH Aachen University, 2016.

\bibitem{KneserApp}
M.~Kneser.
\newblock Strong approximation.
\newblock In {\em Algebraic {G}roups and {D}iscontinuous {S}ubgroups ({P}roc.
  {S}ympos. {P}ure {M}ath., {B}oulder, {C}olo., 1965)}, pages 187--196. Amer.
  Math. Soc., Providence, R.I., 1966.

\bibitem{Kneser}
M.~Kneser.
\newblock {\em Quadratische {F}ormen}.
\newblock Springer-Verlag, Berlin, 2002.
\newblock Revised and edited in collaboration with Rudolf Scharlau.

\bibitem{KMW1}
P.~K{\"o}hler, T.~Meixner, and M.~Wester.
\newblock The {$2$}-adic affine building of type {$\tilde{A}_2$} and its finite
  projections.
\newblock {\em J. Combin. Theory Ser. A}, 38(2):203--209, 1985.

\bibitem{KL}
D.~Lorch and M.~Kirschmer.
\newblock Single-class genera of positive integral lattices.
\newblock {\em LMS J. Comput. Math.}, 16:172--186, 2013.

\bibitem{M}
T.~Meixner.
\newblock Klassische {T}its {K}ammersysteme mit einer transitiven
  {A}utomorphismengruppe.
\newblock {\em Mitt. Math. Sem. Giessen}, (174):x+115, 1986.

\bibitem{MW1}
T.~Meixner and M.~Wester.
\newblock Some locally finite buildings derived from {K}antor's {$2$}-adic
  groups.
\newblock {\em Comm. Algebra}, 14(3):389--410, 1986.

\bibitem{Mu}
D.~Mumford.
\newblock An algebraic surface with {$K$} ample, {$(K^{2})=9$}, {$p_{g}=q=0$}.
\newblock {\em Amer. J. Math.}, 101(1):233--244, 1979.

\bibitem{Nebe}
G.~Nebe.
\newblock Finite quaternionic matrix groups.
\newblock {\em Represent. Theory}, 2:106--223, 1998.

\bibitem{OMeara}
O.~T. O'Meara.
\newblock {\em Introduction to {Q}uadratic {F}orms}.
\newblock Springer, 1973.

\bibitem{Scharlau}
R.~Scharlau.
\newblock Unimodular lattices over real quadratic fields.
\newblock {\em Math. Z.}, 216:437--452, 1994.

\bibitem{ShimuraUnitary}
G.~Shimura.
\newblock Arithmetic of unitary groups.
\newblock {\em Ann. Math.}, 79:269--409, 1964.

\bibitem{Taylor}
D.~E. Taylor.
\newblock {\em The geometry of the classical groups}, volume~9 of {\em Sigma
  Series in Pure Mathematics}.
\newblock Heldermann Verlag, Berlin, 1992.

\bibitem{Tits}
J.~Tits.
\newblock Reductive groups over local fields.
\newblock In {\em Automorphic forms, representations and {$L$}-functions
  ({P}roc. {S}ympos. {P}ure {M}ath.}, volume~33, pages 29--69. Amer. Math.
  Soc., Providence, R.I., 1979.

\bibitem{pMap}
G.~L. Watson.
\newblock Transformations of a quadratic form which do not increase the
  class-number.
\newblock {\em Proc. London Math. Soc. (3)}, 12:577--587, 1962.

\end{thebibliography}

\end{document}